\documentclass[final,onefignum,onetabnum]{siamart190516}

\usepackage{braket,amsfonts}
\usepackage{latexsym,amsmath,amssymb}
\usepackage{array}

\usepackage[caption=false]{subfig}
\usepackage{pgfplots}

\newsiamthm{claim}{Claim}
\newsiamremark{remark}{Remark}
\newsiamremark{hypothesis}{Hypothesis}
\crefname{hypothesis}{Hypothesis}{Hypotheses}

\usepackage{algorithmic}

\usepackage{graphicx,epstopdf}

\Crefname{ALC@unique}{Line}{Lines}

\usepackage{amsopn}

\usepackage{xspace}
\usepackage{bold-extra}
\usepackage[most]{tcolorbox}

\colorlet{texcscolor}{blue!50!black}
\colorlet{texemcolor}{red!70!black}
\colorlet{texpreamble}{red!70!black}
\colorlet{codebackground}{black!25!white!25}

\lstdefinestyle{siamlatex}{%
  style=tcblatex,
  texcsstyle=*\color{texcscolor},
  texcsstyle=[2]\color{texemcolor},
  keywordstyle=[2]\color{texemcolor},
  moretexcs={cref,Cref,maketitle,mathcal,text,headers,email,url},
}

\tcbset{%
  colframe=black!75!white!75,
  coltitle=white,
  colback=codebackground, 
  colbacklower=white, 
  fonttitle=\bfseries,
  arc=0pt,outer arc=0pt,
  top=1pt,bottom=1pt,left=1mm,right=1mm,middle=1mm,boxsep=1mm,
  leftrule=0.3mm,rightrule=0.3mm,toprule=0.3mm,bottomrule=0.3mm,
  listing options={style=siamlatex}
}

\newtcblisting[use counter=example]{example}[2][]{%
  title={Example~\thetcbcounter: #2},#1}

\newtcbinputlisting[use counter=example]{\examplefile}[3][]{%
  title={Example~\thetcbcounter: #2},listing file={#3},#1}

\DeclareTotalTCBox{\code}{ v O{} }
{ 
  fontupper=\ttfamily\color{black},
  nobeforeafter,
  tcbox raise base,
  colback=codebackground,colframe=white,
  top=0pt,bottom=0pt,left=0mm,right=0mm,
  leftrule=0pt,rightrule=0pt,toprule=0mm,bottomrule=0mm,
  boxsep=0.5mm,
  #2}{#1}

\patchcmd\newpage{\vfil}{}{}{}
\begin{tcbverbatimwrite}{tmp_\jobname_header.tex}
\title{Convergence Rates of Inertial Primal-Dual Dynamical  Methods for  Separable Convex Optimization Problems\thanks{The authors certify that the general content of the manuscript, in whole or in part, is not submitted, accepted, or published elsewhere, including conference proceedings.
\funding{This work was supported by  the National Natural Science  Foundation of China
(11471230) and the Scientific Research Foundation  of the Education Department of Sichuan Province (16ZA0213).}}}

\author{Xin He\thanks{Department of Mathematics, Sichuan University, Chengdu, Sichuan, P.R. China(\email{hexinuser@163.com}).}
\and Rong Hu\thanks{Department of Applied Mathematics, Chengdu University of Information Technology, Chengdu, Sichuan, P.R. China (\email{ronghumath@aliyun.com}).}
\and Ya Ping Fang\thanks{Department of Mathematics, Sichuan University, Chengdu, Sichuan, P.R. China(\email{ypfang@aliyun.com}, Corresponding author).}}

\headers{Inertial Primal-Dual Dynamical  Methods}{X. HE, R. HU, and Y.P. FANG}
\end{tcbverbatimwrite}
\input{tmp_\jobname_header.tex}

\ifpdf
\hypersetup{ pdftitle={Convergence Rates of Inertial Primal-Dual Dynamical  Methods for  Separable Convex Optimization Problems},}
\fi

\begin{document}
\maketitle

\begin{abstract}
In this paper, we propose a second-order continuous primal-dual dynamical system with time-dependent positive damping terms for a separable convex optimization problem with linear equality constraints. By the Lyapunov function approach, we investigate asymptotic properties of  the proposed dynamical system  as the time $t\to+\infty$. The convergence rates
are derived for different choices of the damping coefficients. We also show that  the obtained results are robust under external perturbations.
\end{abstract}



\section{Introduction}

\subsection{Problem statement}
Throughout this paper we discuss in the Euclidean spaces with the inner $\langle \cdot,\cdot\rangle$ and the norm $\|\cdot\|$.  Let $f: \mathbb{R}^{n_1}\to\mathbb{R}$ and $g: \mathbb{R}^{n_2}\to\mathbb{R}$ be two smooth convex functions.  Consider the  separable convex optimization problem:
	\begin{eqnarray}\label{eq:ques}
				\min && \ f(x) + g(y)\nonumber\\
				s.t. \, &&  Ax+By = b,
	\end{eqnarray}
where $A\in\mathbb{R}^{m\times n_1}$, $B\in\mathbb{R}^{m\times n_2}$ and $b\in\mathbb{R}^{m}$. This problem plays important roles in diverse applied fields such as, machine learning, signal recovery, structured nonlinear theory and image recovery (see, e.g., \cite{Boyd,ChenL2019,Goldstein2014,Lin2019}).

Denoted by $\Omega$ the KKT point set of the problem \eqref{eq:ques}, i.e., $(x^*,y^*,\lambda^*)\in\Omega$ if and only if 
\begin{equation}\label{eq:saddle_point}
	\begin{cases}
		-A^T\lambda^* = \nabla f(x^*),\\
		-B^T\lambda^* = \nabla g(y^*),\\
		Ax^*+By^* -b =0.
	\end{cases}
\end{equation}
In what follows, we always suppose that  $\Omega\ne\emptyset$. It is well-known that  $(x^*,y^*)$ solves the problem \eqref{eq:ques} if and only if there exists $\lambda^*\in \mathbb{R}^{m}$ such that  $(x^*,y^*,\lambda^*)\in\Omega$. The augment Lagrangian function $\mathcal{L}:\mathbb{R}^{n_1}\times\mathbb{R}^{n_2}\times\mathbb{R}^m\to\mathbb{R}$, associated with the problem \eqref{eq:ques}, is defined by
\begin{equation}\label{eq:AugL}
	\mathcal{L}(x,y,\lambda) = f(x)+g(y)+\langle \lambda,Ax+By-b\rangle+\frac{1}{2}\|Ax+By-b\|^2. 
\end{equation}
Then,  $(x^*,y^*,\lambda^*)\in\Omega$ if and only if it is a saddle point of  $\mathcal{L}$, i.e.,
\[ \mathcal{L}(x^*,y^*,\lambda)\leq  \mathcal{L}(x^*,y^*,\lambda^*)\leq  \mathcal{L}(x,y,\lambda^*), \qquad \forall (x,y,\lambda)\in \mathbb{R}^{n_1}\times\mathbb{R}^{n_2}\times\mathbb{R}^m.\]  

Given a fixed $t_0> 0$, in terms of the augment Lagrangian function $ \mathcal{L}$, we propose the following inertial primal-dual dynamical system for solving the problem \eqref{eq:ques}:
\begin{equation*}
	\begin{cases}
		\ddot{x}(t)+\gamma(t)\dot{x}(t) = -\nabla_x \mathcal{L}(x(t),y(t),\lambda(t)+\delta(t)\dot{\lambda}(t)),\\
		\ddot{y}(t)+\gamma(t)\dot{y}(t) = -\nabla_y  \mathcal{L}(x(t),y(t),\lambda(t)+\delta(t)\dot{\lambda}(t)),\\
		\ddot{\lambda}(t)+\gamma(t)\dot{\lambda}(t) =\nabla_{\lambda}\mathcal{L}(x(t)+\delta(t)\dot{x}(t),y(t)+\delta(t)\dot{y}(t),\lambda(t)),
	\end{cases}
\end{equation*}
where $\gamma,\delta:[t_0,+\infty)\to (0,+\infty)$  are two continuous damping functions. 
By  computations, the inertial primal-dual dynamical system can be rewritten as follows:
\begin{equation}\label{dy:dy_unpertu}  
	\begin{cases}
		\ddot{x}(t)+\gamma(t)\dot{x}(t) = -\nabla f(x(t))-A^T(\lambda(t)+\delta(t)\dot{\lambda}(t))-A^T(Ax(t)+By(t)-b),\\
		\ddot{y}(t)+\gamma(t)\dot{y}(t) = -\nabla g(y(t))-B^T(\lambda(t)+\delta(t)\dot{\lambda}(t))-B^T(Ax(t)+By(t)-b),\\
		\ddot{\lambda}(t)+\gamma(t)\dot{\lambda}(t) = A(x(t)+\delta(t)\dot{x}(t))+B(y(t)+\delta(t)\dot{y}(t))-b.
	\end{cases}
\end{equation}
In this paper we shall discuss the convergence  rate analysis of the proposed  inertial primal-dual dynamical method for the problem \eqref{eq:ques} by  investigating the asymptotic behavior of the  inertial primal-dual dynamical system \eqref{dy:dy_unpertu} as $ t\to +\infty$. 

\subsection{Historical presentation}
In recent years, the second-order dynamical system method is  very popular for solving the unconstrained smooth optimization problem
\begin{equation}\label{eq:min_fun} 
   \min \Phi(x),
\end{equation}
where $\Phi(x)$ is a smooth cost function. To solve the problem \eqref{eq:min_fun},  Polyak \cite{Polyak1964,Polyak1987} introduced the heavy ball with friction system
\begin{equation}\label{dy:hball}
\ddot{x}(t)+\gamma\dot{x}(t)+\nabla \Phi(x(t))=0,
\end{equation}
where  $\gamma>0$ is a damping coefficient. Alvarez \cite{Alvarez2001}  studied the asymptotic behavior of the heavy ball with friction system \eqref{dy:hball} under the condition that $\Phi(x)$ is  convex. The asymptotic behavior of  \eqref{dy:hball} with $\Phi(x)$ being nonconvex was discussed by  B\'{e}gout et al. \cite{Begot2015}.  Haraux and Jendoubi \cite{Haraux2012} investigated the asymptotic behavior of the following  perturbed version of the heavy ball with friction system \eqref{dy:hball}:
\begin{equation}\label{dy:hball-p}
\ddot{x}(t)+\gamma\dot{x}(t)+\nabla \Phi(x(t))=\epsilon(t),
\end{equation}
where  $\epsilon(t)$  is used as a perturbation. When the positive damping coefficient is dependent upon the time $t$,   \eqref{dy:hball} and \eqref{dy:hball-p}  become, respectively, the  following inertial gradient system 
\begin{equation*}
 (IGS_{\gamma})\qquad	\ddot{x}(t)+\gamma(t)\dot{x}(t)+\nabla \Phi(x(t))=0,
\end{equation*}
and its perturbed version
\begin{equation*}
 (IGS_{\gamma,\epsilon})\qquad	\ddot{x}(t)+\gamma(t)\dot{x}(t)+\nabla \Phi(x(t))=\epsilon(t).
\end{equation*}
The importance of $(IGS_{\gamma})$ and  $(IGS_{\gamma,\epsilon})$ has been recognized in the fields of fast optimization methods, control theory, and mechanics. Here, we mention some nice works concerning  fast optimization methods.
Su et al. \cite{Su2014} pointed out that $(IGS)_{\gamma}$ with $\gamma(t)=\frac{3}{t}$ can be viewed as a continuous version of  the Nesterov's accelerated gradient algorithm (see \cite{Nesterov1983,Nesterov2013}). The convergence rate $\Phi(x(t))-\min \Phi=\mathcal{O}(\frac{1}{t^2})$ was also obtained in \cite{Su2014} for  $(IGS)_{\gamma}$ with  $\gamma(t)=\frac{\alpha}{t}$ when  $\alpha\geq 3$.   Attouch et al. \cite{AttouchCP2018} generalized this result  by showing that  $\Phi(x(t))-\min \Phi=\mathcal{O}(\frac{1}{t^2})$   for $(IGS)_{\gamma,\epsilon}$ with $\gamma(t)=\frac{\alpha}{t}$, $\alpha\ge 3$,  and $\epsilon(t)$  satisfying  $\int^{+\infty}_{t_0}t\|\epsilon(t)\|ds<+\infty$. In the case  $\gamma(t)=\frac{\alpha}{t}$ with  $\alpha > 3$,  May \cite{May2015}  proved an improved convergence rate  $\Phi(x(t))-\min \Phi= o(\frac{1}{t^2})$  for  $(IGS)_{\gamma}$. When $\gamma(t)=\frac{\alpha}{t}$ with  $\alpha\leq 3$,  it was shown in   \cite{AttouchCR2019,Vassilis2018} that the convergence rate of the values along the trajectory is $\Phi(x(t))-\min \Phi=\mathcal{O}(t^\frac{-2\alpha}{3}) $ for  $(IGS)_{\gamma}$. In the case  $\gamma(t)=\frac{\alpha}{t}$ and $\alpha>0$,   Aujol et al. \cite{Aujol2019} studied the convergence rate  of the values along the trajectory  under some additional geometrical conditions on $\Phi(x)$.  
When $\gamma(t)=\frac{\alpha}{t^r}$ with $r\in(0,1)$,  Cabot and Frankel \cite{CabotF2012} studied the asymptotic behavior of  $(IGS_{\gamma})$. Jendoubi and May \cite{Jendoubi} extended the results of Cabot and Frankel \cite{CabotF2012} to the perturbed case, and the corresponding convergence rate results can be found in \cite{Balti2016,May2015_2}. The results on asymptotic  behaviors of  $(IGS_{\gamma})$ and  $(IGS_{\gamma,\epsilon})$ with a general damping function $\gamma(t)$ can be found in \cite{AttouchC2017,AttouchCCR2018,AttouchCR2019HAL,CabotEG2007,CabotEG2009}. For more results on  second-order dynamical system approaches for unstrained optimization problems, we refer the reader to \cite{attouch2018comb,BotC2016,BotC2018,Luo2019,Sebbouh2020}.

For the  linear equality constrained optimization problem \eqref{eq:ques}, popular numerical  methods are based on the primal-dual framework (see, e.g., \cite{Beck2017,Boyd,Chambolle2011,Goldstein2014}).  In recent years, some first-order dynamical  system methods based on the primal-dual framework were proposed for solving the  problem
\eqref{eq:ques} (see, e.g.,\cite{CherukuriGC2017,Cherukuri2016,Feijer2010,ZengYP2016}). However, to the best of our knowledge,  second-order dynamical  system methods based on the primal-dual framework are less discussed. It is worth mentioning that  $IGS_{\gamma}$ and  $IGS_{\gamma,\epsilon}$  proposed for unstrained optimization problems cannot be directly applied to the primal-dual framework for the problem \eqref{eq:ques}. Recently, Zeng et al. \cite{ZengJ2019}  proposed the following second-order dynamical system based on the primal-dual framework for solving the problem \eqref{eq:ques} with $g(x)\equiv 0$ and $B = 0$:
\begin{equation*}	
	\begin{cases}
		\ddot{x}(t)+\frac{\alpha}{t}\dot{x}(t) = -\nabla f(x(t))-A^T(\lambda(t)+\beta t\dot{\lambda}(t))-A^T(Ax(t)-b),\\
		\ddot{\lambda}(t)+\frac{\alpha}{t}\dot{\lambda}(t) = A(x(t)+\beta t\dot{x}(t))-b
	\end{cases}
\end{equation*}
and proved  $\mathcal{L}(x(t),\lambda^*)-\mathcal{L}(x^*,\lambda^*)= \mathcal{O}(1/t^{\frac{2}{3}\min\lbrace 3,\alpha\rbrace})$ and  $\|Ax(t)-b\|= \mathcal{O}(1/t^{\frac{\min\lbrace 3,\alpha\rbrace}{3}})$ with $\alpha>0$ and $\beta = \frac{3}{2\min \lbrace 3,\alpha\rbrace}$.

\subsection{Organization}
In Section 2, based on new Lyapunov analysis, we obtain the existence and uniqueness of a global solution and discuss the asymptotic properties of the trajectories generated by the dynamic \eqref{dy:dy_unpertu}  when   $\gamma(t)$ meets certain conditions. The results  covers the ones of the Nesterov's accelerated  gradient system in which $\gamma(t)=\frac{\alpha}{t}$ with $\alpha>0$. In Section 3,  we establish the existence and uniqueness of a global solution and  investigate the asymptotic properties in the case $\gamma(t)=\frac{\alpha}{t^r}$ with $r\in(-1,1)$. Finally, in Section 4, we complement these results by showing that the results obtained are robust with respect to external perturbations.

\section{Asymptotic properties of  \eqref{dy:dy_unpertu} with a general  $\gamma(t)$ }

In this section we discuss the asymptotic behavior of \eqref{dy:dy_unpertu} with a general  $\gamma(t)$ as the time $t\to +\infty$.  To do this, we first establish the existence of a global solution of  the dynamic  \eqref{dy:dy_unpertu}. The following proposition, whose proof follows from the Picard-Lindelof Theorem (see \cite[Theorem 2.2]{Teschl2012}), establishes the existence and uniqueness of a local solution of the dynamic \eqref{dy:dy_unpertu}:
\begin{proposition}\label{pro:local_exist}
	Let $f$ and $g$ be two continuously differentiable functions such that $\nabla f$ and $\nabla g$ are locally Lipschitz continuous, and let $\gamma,\delta:[t_0,+\infty)\to (0,+\infty)$ be locally integrable. Then for any $(x_0,y_0,\lambda_0,u_0,v_0,w_0)$, there exists a unique solution $(x(t),y(t),\lambda(t))$ with $x(t)\in\mathcal{C}^2([t_0,T),\mathbb{R}^{n_1})$,  $y(t)\in\mathcal{C}^2([t_0,T),\mathbb{R}^{n_2})$ and  $\lambda(t)\in\mathcal{C}^2([t_0,T),\mathbb{R}^{m})$ of the dynamic \eqref{dy:dy_unpertu} satisfying $(x(t_0),y(t_0),\lambda(t_0))=(x_0,y_0,\lambda_0)$ and $(\dot{x}(t_0), \dot{y}(t_0),\dot{\lambda}(t_0))=(u_0,v_0,w_0)$ on a maximal interval $[t_0,T)\subset[t_0,+\infty)$.
\end{proposition}

To analyze the asymptotic behavior of  the dynamic \eqref{dy:dy_unpertu}, it is necessary to prove the existence of a global solution. To do so, we  introduce the following function $p:[t_0,+\infty)\to[1,+\infty)$ defined by
\begin{equation}\label{eq:pt}
p(t) = e^{\int_{t_0}^t\gamma(s) ds},\quad \forall t\geq t_0,
\end{equation}
which will be used for  convergence rate analysis. It is easy to verify that
\begin{equation}\label{eq:pt_der}
	\dot{p}(t) = p(t)\gamma(t), \qquad\forall t\geq t_0. 
\end{equation}
Fix $(x^*,y^*,\lambda^*)\in \Omega$. Then  we have $\mathcal{L}(x(t),y(t),\lambda^*)-\mathcal{L}(x^*,y^*,\lambda^* )\geq 0$ for all $t\in [t_0,T)$. Consider the energy function $\mathcal{E}_{\theta,\eta}^{\beta}:[t_0,T)\to[0,+\infty)$ defined by
\begin{equation}\label{eq:energy_fun_gen}
	\mathcal{E}_{\theta,\eta}^{\beta}(t) = \mathcal{E}_0(t)+\mathcal{E}_1(t)+\mathcal{E}_2(t)+\mathcal{E}_3(t),
\end{equation}
where
\begin{equation*}
	\begin{cases}
		\mathcal{E}_0(t) =p(t)^{2\beta}(\mathcal{L}(x(t),y(t),\lambda^*)-\mathcal{L}(x^*,y^*,\lambda^*)	),\\
		\mathcal{E}_1(t) = \frac{1}{2}\|\theta(t)(x(t)-x^*)+p(t)^{\beta}\dot{x}(t)\|^2+\frac{\eta(t)}{2}\|x(t)-x^*\|^2,\\
		\mathcal{E}_2(t) = \frac{1}{2}\|\theta(t)(y(t)-y^*)+p(t)^{\beta}\dot{y}(t)\|^2+\frac{\eta(t)}{2}\|y(t)-y^*\|^2,\\
		\mathcal{E}_3(t) = \frac{1}{2}\|\theta(t)(\lambda(t)-\lambda^*)+p(t)^{\beta}\dot{\lambda}(t)\|^2+\frac{\eta(t)}{2}\|\lambda(t)-\lambda^*\|^2,
	\end{cases}
\end{equation*}
$\theta,\eta:[t_0,+\infty)\to [0,+\infty)$ are two suitable functions, and $\beta$ is a positive constant.

Multiplying the first equation of \eqref{dy:dy_unpertu} by $p(t)^{\beta}$ we get
\[p(t)^{\beta}\ddot{x}(t) = -p(t)^{\beta}(\gamma(t)\dot{x}(t)+\nabla f(x(t))+A^T(\lambda(t)+\delta(t)\dot{\lambda}(t))+A^T(Ax(t)+By(t)-b)), \]
which together with  \eqref{eq:pt_der} yields
\begin{eqnarray*}
	&&\dot{\mathcal{E}}_1(t)=\langle 
\theta(t)(x(t)-x^*)+p(t)^{\beta}\dot{x}(t),	\dot{\theta}(t)(x(t)-x^*)+(\theta(t)+\beta p(t)^{\beta-1}\dot{p}(t))\dot{x}(t)\rangle\\
&&\quad+\langle 
\theta(t)(x(t)-x^*)+p(t)^{\beta}\dot{x}(t),p(t)^{\beta}\ddot{x}(t)\rangle+ \frac{\dot{\eta}(t)}{2}\|x(t)-x^*\|^2\\
&&\quad+ \eta(t) \langle x(t)-x^*,\dot{x}(t)\rangle \\
	&& =  \langle 
\theta(t)(x(t)-x^*)+p(t)^{\beta}\dot{x}(t),	\dot{\theta}(t)(x(t)-x^*)+(\theta(t)+(\beta-1)p(t)^{\beta}\gamma(t))\dot{x}(t)\rangle\\
\quad &&\quad -p(t)^{\beta}\langle \theta(t)(x(t)-x^*)+p(t)^{\beta}\dot{x}(t),	\nabla f(x(t))+A^T(\lambda(t)+\delta(t)\dot{\lambda}(t))\rangle\\
	&&\quad-p(t)^{\beta}\langle \theta(t)(x(t)-x^*)+p(t)^{\beta}\dot{x}(t),   A^T(Ax(t)+By(t)-b)\rangle \\
	&&\quad+\frac{\dot{\eta}(t)}{2}\|x(t)-x^*\|^2+ \eta(t) \langle x(t)-x^*,\dot{x}(t)\rangle \\
    && = (\theta(t)\dot{\theta}(t)+\frac{\dot{\eta}(t)}{2})\|x(t)-x^*\|^2+p(t)^{\beta}(\theta(t)+(\beta-1)p(t)^{\beta}\gamma(t))\|\dot{x}(t)\|^2\\
    &&\quad+(\theta(t)(\theta(t)+(\beta-1)p(t)^{\beta}\gamma(t))+\dot{\theta}(t)p(t)^{\beta}+\eta(t))\langle x(t)-x^*,\dot{x}(t)\rangle\\
	&&\quad - \theta(t)p(t)^{\beta}\langle x(t)-x^*,\nabla f(x(t))+A^T(\lambda(t)+\delta(t)\dot{\lambda}(t))\rangle\\
	&&\quad -\theta(t)p(t)^{\beta}\langle Ax(t)-Ax^*,Ax(t)+By(t)-b\rangle\\
	&&\quad - p(t)^{2\beta}\langle \dot{x}(t),\nabla f(x(t))+A^T(\lambda(t)+\delta(t)\dot{\lambda}(t))+A^T(Ax(t)+By(t)-b)\rangle.\\
\end{eqnarray*}
By similar arguments, we have
\begin{eqnarray*}
	&& \dot{\mathcal{E}}_2(t) =(\theta(t)\dot{\theta}(t)+\frac{\dot{\eta}(t)}{2})\|y(t)-y^*\|^2 + p(t)^{\beta}(\theta(t)+(\beta-1)p(t)^{\beta}\gamma(t))\|\dot{y}(t)\|^2\\
	&& \quad +(\theta(t)(\theta(t)+(\beta-1)p(t)^{\beta}\gamma(t))+\dot{\theta}(t)p(t)^{\beta}+\eta(t))\langle y(t)-y^*,\dot{y}(t)\rangle\\
 	&&\quad - \theta(t)p(t)^{\beta}\langle y(t)-y^*,\nabla g(y(t))+B^T(\lambda(t)+\delta(t)\dot{\lambda}(t))\rangle\\
  	&&\quad - \theta(t)p(t)^{\beta}\langle By(t)-By^*,Ax(t)+By(t)-b\rangle\\
 	&&\quad - p(t)^{2\beta}\langle \dot{y}(t),\nabla g(y(t))+B^T(\lambda(t)+\delta(t)\dot{\lambda}(t))+B^T(Ax(t)+By(t)-b)\rangle
\end{eqnarray*}
and
\begin{eqnarray*}
	&&\dot{\mathcal{E}}_3(t) = (\theta(t)\dot{\theta}(t)+\frac{\dot{\eta}(t)}{2})\|\lambda(t)-\lambda^*\|^2 + p(t)^\beta(\theta(t)+(\beta-1)p(t)^{\beta}\gamma(t))\|\dot{\lambda}(t)\|^2\\
	&&\quad+(\theta(t)(\theta(t)+(\beta-1)p(t)^\beta\gamma(t))+\dot{\theta}(t)p(t)^\beta+\eta(t))\langle \lambda(t)-\lambda^*,\dot{\lambda}(t)\rangle\\
	&&\quad+ \theta(t)p(t)^\beta\langle \lambda(t)-\lambda^*,A(x(t)+\delta(t)\dot{x}(t))+B(y(t)+\delta(t)\dot{y}(t))-b\rangle\\
	&&\quad + p(t)^{2\beta}\langle \dot{\lambda}(t),A(x(t)+\delta(t)\dot{x}(t))+B(y(t)+ \delta(t)\dot{y}(t))-b
\rangle.
\end{eqnarray*}
Adding $\dot{\mathcal{E}}_1(t)$, $\dot{\mathcal{E}}_2(t)$, $\dot{\mathcal{E}}_3(t)$ together, using $Ax^*+By^*= b$ and rearranging the terms, we have
\[	\dot{\mathcal{E}}_1(t)+\dot{\mathcal{E}}_2(t)+\dot{\mathcal{E}}_3(t) = \sum^5_{i=1} \mathcal{V}_i(t),\] 
where
\begin{eqnarray*}
	 \mathcal{V}_1(t) &=&  \left(\theta(t)\dot{\theta}(t)+\frac{\dot{\eta}(t)}{2}\right)(\|x(t)-x^*\|^2+\|y(t)-y^*\|^2+\|\lambda(t)-\lambda^*\|^2),\\
	 \mathcal{V}_2(t) &=&(\theta(t)(\theta(t)+(\beta-1)p(t)^\beta\gamma(t))+\dot{\theta}(t)p(t)^\beta+\eta(t))\\
	 &&\times(\langle x(t)-x^*,\dot{x}(t)\rangle+\langle y(t)-y^*,\dot{y}(t)\rangle+\langle \lambda(t)-\lambda^*,\dot{\lambda}(t)\rangle),\\
	 \mathcal{V}_3(t) &=& -\theta(t)p(t)^\beta(\langle x(t)-x^*, \nabla f(x(t))+A^T\lambda^*\rangle+\langle y(t)-y^*, \nabla g(y(t))+B^T\lambda^*\rangle)\\
		\quad &&+ \theta(t)p(t)^\beta\delta(t) \langle \lambda(t)-\lambda^*, A\dot{x}(t)+B\dot{y}(t)\rangle,
\\
	 \mathcal{V}_4(t) &=&p(t)^\beta(\theta(t)+(\beta-1)p(t)^\beta\gamma(t))(\|\dot{x}(t)\|^2+\|\dot{y}(t)\|^2+\|\dot{\lambda}(t)\|^2)\\
	 && -\theta(t)p(t)^\beta\|Ax(t)+By(t)-b\|^2,\\
	 \mathcal{V}_5(t) &=&(p(t)^{2\beta}-\theta(t) p(t)^\beta\delta(t))\langle \dot{\lambda}(t),Ax(t)+By(t)-b\rangle\\
	\quad && -p(t)^{2\beta}\langle \dot{x}(t),\nabla f(x(t))+A^T\lambda(t)+A^T(Ax(t)+By(t)-b)\rangle\\
	\quad && -p(t)^{2\beta}\langle \dot{y}(t),\nabla g(y(t))+B^T\lambda(t)+B^T(Ax(t)+By(t)-b)\rangle.
\end{eqnarray*}
Derivate $\mathcal{E}_0(t)$ to get
\begin{eqnarray*}
	&&\dot{\mathcal{E}}_0(t)=2\beta p(t)^{2\beta}\gamma(t)(f(x(t))-f(x^*)+g(y(t))-g(y^*)+\langle \lambda^*,Ax(t)+By(t)-b\rangle)\\
	&&\quad+\beta p(t)^{2\beta}\gamma(t)\|Ax(t)+By(t)-b\|^2 + p(t)^{2\beta}(\langle \nabla f(x(t)),\dot{x}(t)\rangle+\langle \nabla g(y(t)),\dot{y}(t)\rangle)\\
	&&\quad+p(t)^{2\beta}(\langle \lambda^*, A\dot{x}(t)+B\dot{y}(t)\rangle+\langle Ax(t)+By(t)-b,A\dot{x}(t)+B\dot{y}(t)\rangle).
\end{eqnarray*}

Now we are in a position to investigate the existence and uniqueness of a global solution of the dynamic \eqref{dy:dy_unpertu} with suitable choices of $\gamma(t)$ and $\delta(t)$. 
\begin{theorem}\label{th:th_exist_gen}
	Let $f$ and $g$ be two continuously differentiable functions such that $\nabla f$ and $\nabla g$ are locally Lipschitz continuous, $\gamma:[t_0,+\infty)\to(0,+\infty)$ be a nonincreasing  and twice continuously differentiable function satisfying
			\[ \ddot{\gamma}(t)\geq 2\beta^2 \gamma(t)^3,    \qquad\forall t\geq t_0 \] for some $\beta\in(0,\frac{1}{3})$, and $\delta(t)=\frac{1}{\beta_0\gamma(t)}$ with $\beta_0\in[2\beta,1-\beta)$.
Let $(x^*,y^*,\lambda^*)\in \Omega$ and  $(x(t),y(t),\lambda(t))$ be the unique solution of the dynamic \eqref{dy:dy_unpertu}  defined on a maximal interval $[t_0,T)$ with $T\leq +\infty$ for some initial value.  Then, the following conclusions hold:
	\begin{itemize}
		\item [(a)]  There exist positive functions $\theta(t)$ and $\eta(t)$ satisfying
		\[\dot{\mathcal{E}}_{\theta,\eta}^{\beta}(t)\leq 0,\quad \forall t\in[t_0,T). \]
		As a consequence,  the function  $\mathcal{E}_{\theta,\eta}^{\beta}(t)$  is nonincreasing on $[t_0 , T)$.
		\item [(b)] $T=+\infty$
.
\end{itemize}
\end{theorem}
\begin{proof}
(a):  
Take
\begin{equation}\label{eq:theta_eta_gen}
	\theta(t) = \beta_0 p(t)^\beta\gamma(t)  \quad\text{and}\quad \eta(t) = -\beta_0 p(t)^{2\beta}((\beta_0+2\beta-1)\gamma(t)^2+\dot{\gamma}(t)).
	\end{equation}
Clearly, $\theta(t)>0$ for all $t\ge t_0$.  By assumption, we have 
\[ \ddot{\gamma}(t)\geq 2\beta^2\gamma(t)^3, \quad\forall t\geq t_0.\]
This together with  Lemma \ref{le:A1} yields
\begin{equation}\label{eq:gamma_sec_der}
	 \dot{\gamma}(t)\leq -\beta\gamma(t)^2, \quad \forall t\geq t_0.
\end{equation}
Since   $\beta\in(0,\frac{1}{3})$ and $\beta_0\in[2\beta,1-\beta)$, it follows from \eqref{eq:theta_eta_gen} and \eqref{eq:gamma_sec_der} that 
\begin{equation}\label{eq:eta_gen}
	\eta(t) \geq  \beta_0(1-\beta-\beta_0) p(t)^{2\beta}\gamma(t)^2> 0
\end{equation}
for $t\geq t_0$. By computations, we have 
\begin{equation}\label{eq:theta_eta_gen_der}
\begin{cases}
	\dot{\theta}(t) = \beta_0 p(t)^\beta (\beta\gamma(t)^2+\dot{\gamma}(t)),\\
	\dot{\eta}(t) = -\beta_0 p(t)^{2\beta}(2\beta(\beta_0+2\beta-1)\gamma(t)^3+(6\beta+2\beta_0-2)\gamma(t)\dot{\gamma}(t)+\ddot{\gamma}(t)). 
\end{cases}
\end{equation}

We shall prove that for any  $t\geq t_0$,
\begin{eqnarray}
	 \theta(t)\dot{\theta}(t)+\frac{\dot{\eta}(t)}{2} &\leq & 0, \label{eq:cond_gen_1}\\
	 p(t)^{2\beta} - \theta(t)p(t)^{\beta}\delta(t) &=& 0\label{eq:cond_gen_2},\\
	\theta(t)(\theta(t)+(\beta-1)p(t)^\beta\gamma(t))+\dot{\theta}(t)p(t)^\beta+\eta(t) &=& 0.\label{eq:cond_gen_3}
	\end{eqnarray}
Using \eqref{eq:theta_eta_gen} and \eqref{eq:theta_eta_gen_der}, by simple computations  we get \eqref{eq:cond_gen_2} and  \eqref{eq:cond_gen_3}. Next, we shall show \eqref{eq:cond_gen_1}.
Again from  \eqref{eq:theta_eta_gen} and \eqref{eq:theta_eta_gen_der} we have
\begin{eqnarray*}
	\theta(t)\dot{\theta}(t)+\frac{\dot{\eta}(t)}{2} &=& -\frac{\beta_0}{2} p(t)^{2\beta}(\ddot{\gamma}(t)+(6\beta-2)\gamma(t)\dot{\gamma}(t)+2\beta(2\beta-1)\gamma(t)^3)\\
	&=&  -\frac{\beta_0}{2} p(t)^{2\beta}(\ddot{\gamma}(t)-2\beta^2\gamma(t)^3+2(3\beta-1)\gamma(t)(\dot{\gamma}(t)+\beta\gamma(t)^2)).
\end{eqnarray*}
This together with \eqref{eq:gamma_sec_der} and assumption yields   \eqref{eq:cond_gen_1} since  $\beta\in(0,\frac{1}{3})$ and $\beta_0\in[2\beta,1-\beta)$.  Thus, we have proved \eqref{eq:cond_gen_1} - \eqref{eq:cond_gen_3}.

From \eqref{eq:cond_gen_1} and  \eqref{eq:cond_gen_3}  we have $\mathcal{V}_1(t)\leq 0$ and $ \mathcal{V}_2(t)=0$ for any $t\in[t_0,T)$. Since  $f$ and $g$ are convex, it follows from \eqref{eq:cond_gen_2} that
\begin{eqnarray}\label{eq:energy_der_gen}
	&&\dot{\mathcal{E}}_{\theta,\eta}^{\beta}(t)\leq\dot{\mathcal{E}}_0(t)+ \mathcal{V}_3(t)+\mathcal{V}_4(t)+\mathcal{V}_5(t)\nonumber \\
	 &&\quad= \beta_0 p(t)^{2\beta}\gamma(t)(f(x(t))-f(x^*)-\langle x(t)-x^*, \nabla f(x(t))\rangle\nonumber \\
	  &&\qquad+\beta_0 p(t)^{2\beta}\gamma(t)(g(y(t))-g(y^*)-\langle y(t)-y^*, \nabla g(y(t))\rangle) \nonumber\\
	 &&\qquad -  (\beta_0-2\beta)p(t)^{2\beta}\gamma(t)(\mathcal{L}(x(t),y(t),\lambda^*)-\mathcal{L}(x^*,y^*,\lambda^*))\nonumber \\
	 &&\qquad -\frac{\beta_0}{2} p(t)^{2\beta}\gamma(t)\|Ax(t)+By(t)-b\|^2 \\
	 &&\qquad-(1-\beta-\beta_0)   p(t)^{2\beta}\gamma(t)(\|\dot{x}(t)\|^2+\|\dot{y}(t)\|^2+\|\dot{\lambda}(t)\|^2)\nonumber  \\
	 &&\quad\leq -  (\beta_0-2\beta)p(t)^{2\beta}\gamma(t)(\mathcal{L}(x(t),y(t),\lambda^*)-\mathcal{L}(x^*,y^*,\lambda^*))\nonumber \\
	 && \qquad-\frac{\beta_0}{2} p(t)^{2\beta}\gamma(t)\|Ax(t)+By(t)-b\|^2 \nonumber\\
	 &&\qquad-(1-\beta-\beta_0)   p(t)^{2\beta}\gamma(t)(\|\dot{x}(t)\|^2+\|\dot{y}(t)\|^2+\|\dot{\lambda}(t)\|^2)  \nonumber \\
	 &&\quad\leq 0  \nonumber
	\end{eqnarray}
for any $t\in[t_0,T)$.  As a consequence,  the function  $\mathcal{E}_{\theta,\eta}^{\beta}(t)$  is nonincreasing on $[t_0, T)$.

(b): By (a),  $\mathcal{E}_{\theta,\eta}^{\beta}(t)$  is nonincreasing on $[t_0 , T)$. Then,
\begin{equation*}\label{eq:energy_decrease_gen}
	\mathcal{E}_{\theta,\eta}^{\beta}(t)\leq  \mathcal{E}_{\theta,\eta}^{\beta}(t_0),\quad \forall t\in [t_0,T).
\end{equation*}
This implies that $\mathcal{E}_{\theta,\eta}^{\beta}(\cdot)$ is bounded on $[t_0,T)$. 
It follows from \eqref{eq:energy_fun_gen} that
\[\frac{1}{2}\|\theta(t)(x(t)-x^*)+p(t)^{\beta}\dot{x}(t)\|^2 +\frac{\eta(t)}{2}\|x(t)-x^*\|^2 \leq   \mathcal{E}_{\theta,\eta}^{\beta}(t_0),\quad\forall t\in[t_0,T),
 \]
 where $\theta(t)$ and $\eta(t)$ are defined by \eqref{eq:theta_eta_gen}. This implies that
\begin{equation}\label{eq:bound_x_exist_gen}
\eta(t)\|x(t)-x^*\|^2 \leq   2\mathcal{E}_{\theta,\eta}^{\beta}(t_0),\quad \forall t\in[t_0,T)
\end{equation}
and
\begin{equation}\label{eq:bound_dotx_exist_gen}
\|\theta(t)(x(t)-x^*)+p(t)^{\beta}\dot{x}(t)\| \leq \sqrt{ 2\mathcal{E}_{\theta,\eta}^{\beta}(t_0)},\quad\forall t\in[t_0,T).
\end{equation}
Combining  \eqref{eq:bound_x_exist_gen} with   \eqref{eq:eta_gen} we get
 \begin{equation*}
 	 \beta_0(1-\beta-\beta_0) p(t)^{2\beta}\gamma(t)^2\|x(t)-x^*\|^2\leq    2\mathcal{E}_{\theta,\eta}^{\beta}(t_0),\quad \forall t\in[t_0,T), 
 \end{equation*}
which yields
\[\sup_{t\in[t_0,T)}p(t)^{\beta}\gamma(t)\|x(t)-x^*\|<+\infty. \]
It follows from \eqref{eq:bound_dotx_exist_gen} and \eqref{eq:theta_eta_gen} that
\[p(t)^{\beta}\|\dot{x}(t)\|\leq \sqrt{2\mathcal{E}_{\theta,\eta}^{\beta}(t_0)}+\beta_0 p(t)^{\beta}\gamma(t)\|x(t)-x^*\|,\quad\forall  t\in[t_0,T).\]
Since $p(t)\geq 1$, we have 
\[\sup_{t\in[t_0,T)}\|\dot{x}(t)\|\leq  \sqrt{2\mathcal{E}_{\theta,\eta}^{\beta}(t_0)}+\beta_0\sup_{t\in[t_0,T)}p(t)^{\beta}\gamma(t)\|x(t)-x^*\|<+\infty. \]
By similar arguments, we have 
\[\sup_{t\in[t_0,T)}\|\dot{y}(t)\|< +\infty\quad \text{and}\quad   \sup_{t\in[t_0,T)}\|\dot{\lambda}(t)\|<+\infty. \] 
Assume on the contrary that $T<+\infty$. Clearly, the trajectory $(x(t),y(t),\lambda(t))$ is  bounded on $[t_0 , T)$. By assumption and  \eqref{dy:dy_unpertu}, $(\ddot{x}(t),\ddot{y}(t),\ddot{\lambda}(t))$ are bounded on  $[t_0 , T)$. It ensues that  both  $(x(t),y(t),\lambda(t))$  and its derivative  $(\dot{x}(t),\dot{y}(t),\dot{\lambda}(t))$ have a limit at $t = T$, and therefore can be continued, a contradiction. 	Thus $T = +\infty$.
\end{proof}

\begin{remark}
   To establish the existence and uniqueness of a global  solution of  \eqref{dy:dy_unpertu},  it is assumed  in Theorem \ref{th:th_exist_gen} that
\[ \ddot{\gamma}(t)\geq 2\beta^2 \gamma(t)^3,    \quad \forall  t\geq t_0 \] for some $\beta\in(0,\frac{1}{3})$, and $\delta(t)=\frac{1}{\beta_0\gamma(t)}$ with $\beta_0\in[2\beta,1-\beta)$.
From the proof, it is easy to see that the conclusion (b) of Theorem \ref{th:th_exist_gen} still holds if $\beta=\frac{1}{3}$ and $\delta(t)=\frac{3}{2\gamma(t)}$. Under this condition,
\[ \ddot{\gamma}(t)\geq  \frac{2}{9}\gamma(t)^3\geq 2\hat{\beta}^2\gamma(t)^3,    \quad \forall \hat{\beta}\in(0,\frac{1}{3}),\ t\geq t_0 \] 
 and $\delta(t)=\frac{3}{2\gamma(t)}=\frac{1}{\beta_0 \gamma(t)}$ with $\beta_0=\frac{2}{3}\in[2\hat{\beta},1-\hat{\beta})$ for any $\hat{\beta}\in(0,\frac{1}{3})$. 
Let us mention that the condition 
\begin{equation}\label{eq:inqe_ddot_gamma}
\ddot{\gamma}(t)\geq 2\beta^2 \gamma(t)^3\quad\text{ for some  }\quad\beta\in(0,\frac{1}{3}]
\end{equation}
 has been used  in \cite{AttouchCCR2018} for the asymptotic analysis of $(IGS)_{\gamma}$ associated with the unconstrained optimization problem \eqref{eq:min_fun}. As pointed out in \cite{AttouchCCR2018}, the value $\beta=\frac{1}{3}$ is crucial and it corresponds to $\alpha=3$ in the case $\gamma(t)=\frac{\alpha}{t}$. To the best of our knowledge, this is the first time that this condition is applied to the study of  primal-dual dynamical systems for constrained optimization  problems. 
\end{remark}

\begin{remark}
	The existence and uniqueness of a global solution for $IGS_{\gamma}$ associated with the unconstrained optimization problem \eqref{eq:ques} has been established  in  \cite[Proposition 3.2]{AttouchC2017}). The nonincreasing property of the energy function $W(t):=\frac{1}{2}\|\dot{x}(t)\|^2+\Phi(x(t))$ on $[t_0, T )$ plays a crucial role in the proof of  \cite[Proposition 3.2]{AttouchC2017}).   As a comparison,  in Theorem \ref{th:th_exist_gen}  we use the nonincreasing property of the energy function $\mathcal{E}_{\theta,\eta}^{\beta}(t)$  to prove the existence and uniqueness of a global solution for the dynamic \eqref{dy:dy_unpertu}.
\end{remark}


With Theorem \ref{th:th_exist_gen} in hands, we start to discuss the asymptotic behavior of  the dynamic \eqref{dy:dy_unpertu}. The following condition on the damp function $\gamma(t)$ is a common assumption for convergence analysis:
\begin{equation}\label{eq:int_gamma}
{\int_{t_0}^{+\infty}\gamma(t) dt}= +\infty.
\end{equation}
Notice $p(t)=e^{\int_{t_0}^t\gamma(s) ds}\to +\infty$ as $t\to+\infty$  when $\gamma(t)$ satisfies \eqref{eq:int_gamma}.

\begin{theorem}\label{th:rate_gen}
	Let $\gamma:[t_0,+\infty)\to(0,+\infty)$ be a nonincreasing  and twice continuously differentiable function satisfying   \eqref{eq:inqe_ddot_gamma} and  \eqref{eq:int_gamma},
and  $\delta(t)=\frac{1}{\beta_0\gamma(t)}$ with $\beta_0\in[2\beta,1-\beta]$. Suppose that $(x(t),y(t),\lambda(t))$ is a global solution of the dynamic \eqref{dy:dy_unpertu} and  $(x^*,y^*,\lambda^*)\in \Omega$.  Then, the following conclusions hold:
\begin{itemize}
		\item [($a$)]	$\mathcal{L}(x(t),y(t),\lambda^*)-\mathcal{L}(x^*,y^*,\lambda^* )= \mathcal{O}(p(t)^{-2\beta}).$
		\item[($b$)] $ \|Ax(t)+By(t)-b\|= \mathcal{O}(p(t)^{-\beta}).	$
		\item [($c$)] $\int^{+\infty}_{t_0} p(t)^{2\beta}\gamma(t)\|Ax(t)+By(t)-b\|^2 dt <+\infty. $
\end{itemize}
Moreover, we have the following results:
\begin{itemize}
\item[Case I]: $\beta<\frac{1}{3}$ and $\beta_0\in(2\beta,1-\beta)$.  Then
\begin{itemize}
	 \item [($d$)]  $\int^{+\infty}_{t_0}p(t)^{2\beta}\gamma(t)(\mathcal{L}(x(t),y(t),\lambda^*)-\mathcal{L}(x^*,y^*,\lambda^*))dt <+\infty.$
	\item [($e$)]$\int^{+\infty}_{t_0}p(t)^{2\beta}\gamma(t)(\|\dot{x}(t)\|^2+\|\dot{y}(t)\|^2+\|\dot{\lambda}(t)\|^2) dt <+\infty. $
	\item [($f$)] $\|\dot{x}(t)\|+ \|\dot{y}(t)\|+ \|\dot{\lambda}(t)\| =  \mathcal{O}(p(t)^{-\beta})).$

\end{itemize}
 \item[Case II]: $\beta=\frac{1}{3}$ and  $\beta_0=\frac{2}{3}$. Then for any $ \tau\in(0,\frac{1}{3})$ we have
\begin{itemize}
	 \item [($d'$)]  $\int^{+\infty}_{t_0}p(t)^{2\tau}\gamma(t)(\mathcal{L}(x(t),y(t),\lambda^*)-\mathcal{L}(x^*,y^*,\lambda^*))dt <+\infty.$
	\item [($e'$)]$\int^{+\infty}_{t_0}p(t)^{2\tau}\gamma(t)(\|\dot{x}(t)\|^2+\|\dot{y}(t)\|^2+\|\dot{\lambda}(t)\|^2) dt <+\infty. $	 
	\item [($f'$)] $ \|\dot{x}(t)\|+ \|\dot{y}(t)\|+ \|\dot{\lambda}(t)\| =  \mathcal{O}(p(t)^{-\tau}).$
\end{itemize}
\end{itemize}
\end{theorem}
 
 \begin{proof}
Take $\theta(t)$ and $\eta(t)$ as in \eqref{eq:theta_eta_gen}.  Consider the energy function $\mathcal{E}_{\theta,\eta}^{\beta}:[t_0,+\infty)$ $\to[0,+\infty)$ defined by \eqref{eq:energy_fun_gen}. From \eqref{eq:energy_der_gen}, we have  
\begin{equation}\label{eq:th_rate_energy_decrease}
	\mathcal{E}_{\theta,\eta}^{\beta}(t)\leq  \mathcal{E}_{\theta,\eta}^{\beta}(t_0),\quad \forall t\geq t_0
\end{equation}
and
 \begin{eqnarray}\label{eq:th_rate_der_energy}
	&&\dot{\mathcal{E}}_{\theta,\eta}^{\beta}(t)+  (\beta_0-2\beta)p(t)^{2\beta}
	\gamma(t)(\mathcal{L}(x(t),y(t),\lambda^*)-\mathcal{L}(x^*,y^*,\lambda^*))\nonumber \\
	&&\quad+\frac{\beta_0}{2} p(t)^{2\beta}\gamma(t)\|Ax(t)+By(t)-b\|^2
\\
	&&\quad +(1-\beta-\beta_0)   p(t)^{2\beta}\gamma(t)(\|\dot{x}(t)\|^2+\|\dot{y}(t)\|^2+\|\dot{\lambda}(t)\|^2)   \leq 0,\quad \forall t\geq t_0.\nonumber	
\end{eqnarray}
As a consequence of \eqref{eq:th_rate_energy_decrease}, $\mathcal{E}_{\theta,\eta}^{\beta}(\cdot)$ is bounded on $[t_0,+\infty)$.  This together with  \eqref{eq:energy_fun_gen} implies
\begin{equation}\label{eq:rate_L_gen}
	\mathcal{L}(x(t),y(t),\lambda^*)-\mathcal{L}(x^*,y^*,\lambda^*) =\mathcal{O}(p(t)^{-2\beta}).	
\end{equation}
Since $f$ and $g$ are convex, it follows from from \eqref{eq:saddle_point} that
\begin{eqnarray}\label{eq:th_gen_constain}
	 && \mathcal{L}(x(t),y(t),\lambda^*)-\mathcal{L}(x^*,y^*,\lambda^*)\nonumber\\
	 && \quad = f(x(t))-f(x^*)+g(y(t))-g(y^*)+\langle \lambda^*,Ax(t)+By(t)-b\rangle\nonumber \\
	 &&\qquad+\frac{1}{2}\|Ax(t)+By(t)-b\|^2\\
	 	 && \quad = f(x(t))-f(x^*)-\langle -A^T\lambda^*, x(t)-x^*\rangle\nonumber \\
	 	 && \qquad+g(y(t))-g(y^*)-\langle -B^T\lambda^*, y(t)-y^*\rangle+\frac{1}{2}\|Ax(t)+By(t)-b\|^2\nonumber\\
	 	&&\quad \geq \frac{1}{2}\|Ax(t)+By(t)-b\|^2.\nonumber
\end{eqnarray}
This together with \eqref{eq:rate_L_gen} yields
\[ \|Ax(t)+By(t)-b\|= \mathcal{O}(p(t)^{-\beta}).\]
Since $\beta\in(0,\frac{1}{3}]$ and $\beta_0\in[2\beta,1-\beta]$, again from \eqref{eq:th_rate_der_energy} we have 
\[\int^{+\infty}_{t_0}  p(t)^{2\beta}\gamma(t)\|Ax(t)+By(t)-b\|^2 dt <+\infty, \]
\begin{eqnarray} \label{eq:th_gen_integ_1}
	 (\beta_0-2\beta)\int^{+\infty}_{t_0}p(t)^{2\beta}\gamma(t)(\mathcal{L}(x(t),y(t),\lambda^*)-\mathcal{L}(x^*,y^*,\lambda^*))dt <+\infty,
\end{eqnarray}
and
\begin{eqnarray}\label{eq:th_gen_integ_2}
(1-\beta-\beta_0)\int^{+\infty}_{t_0}  p(t)^{2\beta}\gamma(t)(\|\dot{x}(t)\|^2+\|\dot{y}(t)\|^2+\|\dot{\lambda}(t)\|^2) dt <+\infty.
\end{eqnarray}
Thus we have shown $(a)-(c)$.

Next we prove $(d)$,$(e)$ and $(f)$ in the case  $\beta<\frac{1}{3}$ and $\beta_0\in(2\beta,1-\beta)$. Clearly, $\beta_0-2\beta>0$ and $1-\beta-\beta_0>0$.  So $(d)$ and $(e)$ follow directly from \eqref{eq:th_gen_integ_1} and \eqref{eq:th_gen_integ_2}, respectively.

As shown in the proof of $(b)$ of Theorem \ref{th:th_exist_gen}, we have 
\[ \sup_{t\geq t_0}p(t)^{\beta}\gamma(t)\|x(t)-x^*\|<+\infty \]
and
\[p(t)^{\beta}\|\dot{x}(t)\|\leq \sqrt{2\mathcal{E}_{\theta,\eta}^{\beta}(t_0)}+\beta_0 p(t)^{\beta}\gamma(t)\|x(t)-x^*\|,\quad  \forall t\geq t_0.\]
Then
\[\sup_{t\geq t_0} p(t)^{\beta}\|\dot{x}(t)\|\leq \sqrt{2\mathcal{E}_{\theta,\eta}^{\beta}(t_0)}+\beta_0 \sup_{t\geq t_0}p(t)^{\beta}\gamma(t)\|x(t)-x^*\| < +\infty,\]
this implies 
 \[ \|\dot{x}(t)\|= \mathcal{O}(p(t)^{-\beta}).\]
By similar arguments, we have
 \[ \|\dot{y}(t)\|= \mathcal{O}(p(t)^{-\beta})\quad\text{and}\quad  \|\dot{\lambda}(t)\|= \mathcal{O}(p(t)^{-\beta}).\]
This proves $(f)$.

In the case $\beta =\frac{1}{3}$ and $\beta_0=\frac{2}{3}$. For any $\tau\in(0,\frac{1}{3})$, we have $\ddot{\gamma}(t)\geq 2\tau^2\gamma(t)^3$ and $\beta_0\in (2\tau,1-\tau)$. So $(d'),(e')$ and $(f')$ follow directly from $(d)$,$(e)$ and $(f)$, respectively.
\end{proof}

\begin{remark}
It is assumed in Theorem \ref{th:rate_gen} that  $\beta_0\in[2\beta,1-\beta]$. In fact, for any $\beta_0\in (0,1)$,  we can prove  convergence rates as in Theorem \ref{th:rate_gen} by substituting $\bar{\beta}$ for $\beta$, where $\bar{\beta} = \min\lbrace \beta, \frac{\beta_0}{2},1-\beta_0\rbrace$. It is easy to verify that
\[ \ddot{\gamma}(t)\geq 2\bar{\beta}^2 \gamma(t)^3,\quad \bar{\beta}\in(0,\frac{1}{3}],\quad \text{ and }\quad\beta_0\in[2\bar{\beta}, 1-\bar{\beta}].\]
\end{remark}
\begin{remark}
Theorem \ref{th:rate_gen} can be viewed as  analogs of the results  in \cite[Theorem 2.1, Proposition 3, Proposition 4]{AttouchCCR2018}, where  the convergence rate analysis of $(IGS_{\gamma})$  associated with the unconstrained optimization problem \eqref{eq:min_fun}  were derived. In \cite[Theorem 2.1]{AttouchCCR2018}, they assumed that $x(t)$ is bounded on $[t_0,+\infty)$ to get $\|\dot{x}(t)\| =  \mathcal{O}(p(t)^{-\beta})$. Theorem \ref{th:rate_gen} shows that the  boundedness assumption is redundant both in the $IGS_\gamma$ and in our primal-dual dynamical system.
\end{remark}


In the rest of this section, we apply the results of Theorem \ref{th:rate_gen}  to two special damping functions:   $\gamma(t) = \frac{\alpha}{t}$ with $\alpha >0$ and  $\gamma(t)=\frac{1}{t (\ln t)^{r}}$ with $r\in[0,1]$.  

\textbf{Case  $\gamma(t) = \frac{\alpha}{t}$ with $\alpha >0$}. In this case, $\ddot{\gamma}(t) = \frac{2\alpha}{t^3}$ and
\[\ddot{\gamma}(t)\geq 2\beta^2\gamma(t)^3 \Longleftrightarrow \alpha\beta\leq 1. \]
Assumption on $\gamma(t)$ in Theorem \ref{th:rate_gen} is satisfied if we take 
\[ 0<\beta \leq \min\left\lbrace\frac{1}{3}, \frac{1}{\alpha}\right\rbrace. \]

\begin{corollary}\label{cor:cor_alpha_t}
Suppose that  $\gamma(t)=\frac{\alpha}{t}$ with  $\alpha>0$ and  $\delta(t) = \frac{t}{\beta_0\alpha}$ with  $\beta_0>0$. Let  $(x^*,y^*,\lambda^*)\in \Omega$ and  $(x(t),y(t),\lambda(t))$ be a global solution of the dynamic \eqref{dy:dy_unpertu}. 
Then we have the following results:
\begin{itemize}
 \item[i)] If $\alpha\leq 3$ and $\beta_0=\frac{2}{3}$, then
 \begin{itemize}
	 \item [(a)]	$\mathcal{L}(x(t),y(t),\lambda^*)-\mathcal{L}(x^*,y^*,\lambda^* )= \mathcal{O}(t^{-\frac{2\alpha}{3}})$.
		\item[(b)] $ \|Ax(t)+By(t)-b\|= \mathcal{O}(t^{-\frac{\alpha}{3}}).	$
		\item [(c)] $\int^{+\infty}_{t_0} t^{\frac{2\alpha}{3}-1}\|Ax(t)+By(t)-b\|^2 dt <+\infty. $ 
		\item [(d)]  $\int^{+\infty}_{t_0}t^{m}(\mathcal{L}(x(t),y(t),\lambda^*)-\mathcal{L}(x^*,y^*,\lambda^*))dt <+\infty, \quad\forall m\in(-1,\frac{2\alpha}{3}-1).$ 
		\item [(e)] $\int^{+\infty}_{t_0}t^m(\|\dot{x}(t)\|^2+\|\dot{y}(t)\|^2+\|\dot{\lambda}(t)\|^2) dt <+\infty,\quad\forall m\in(-1,\frac{2\alpha}{3}-1)$ .
		\item  [(f)] $\|\dot{x}(t)\|+ \|\dot{y}(t)\|+ \|\dot{\lambda}(t)\| =  \mathcal{O}(t^{-m}), \quad \forall m\in(0,\frac{\alpha}{3})$.
		\end{itemize}
	 Moreover if $\alpha<3$, then
	 \begin{itemize}
		\item[(g)]$\|\dot{x}(t)\|+ \|\dot{y}(t)\|+ \|\dot{\lambda}(t)\| =  \mathcal{O}(t^{-\frac{\alpha}{3}})$.
		\end{itemize}
\item[ii)] If $\alpha > 3$ and  $\beta_0\in (\frac{2}{\alpha},1-\frac{1}{\alpha})$, then
	\begin{itemize}
		\item [(a')]	$\mathcal{L}(x(t),y(t),\lambda^*)-\mathcal{L}(x^*,y^*,\lambda^* )= \mathcal{O}(t^{-2}).$
		\item[(b')] $ \|Ax(t)+By(t)-b\|= \mathcal{O}(t^{-1}).	$
		\item [(c')] $\int^{+\infty}_{t_0}t\|Ax(t)+By(t)-b\|^2 dt <+\infty. $
		\item [(d')]  $\int^{+\infty}_{t_0}t (\mathcal{L}(x(t),y(t),\lambda^*)-\mathcal{L}(x^*,y^*,\lambda^*))dt <+\infty.$
		\item [(e')] $\int^{+\infty}_{t_0}t(\|\dot{x}(t)\|^2+\|\dot{y}(t)\|^2+\|\dot{\lambda}(t)\|^2) dt <+\infty. $
		\item[(f')]$\|\dot{x}(t)\|+ \|\dot{y}(t)\|+ \|\dot{\lambda}(t)\| =  \mathcal{O}(t^{-1})$.
\end{itemize}
\end{itemize}
\end{corollary} 
\begin{proof}
Since $\gamma(t) =\frac{\alpha}{t}$, by computation we have
\begin{equation}\label{col:gen_pt}
p(t) = e^{\int^t_{t_0}\gamma(s)ds}=\left(\frac{t}{t_0}\right)^{\alpha}.
\end{equation}
Take $\beta = \min\left\lbrace\frac{1}{3}, \frac{1}{\alpha}\right\rbrace $. It is easy to verify that all the assumptions in Theorem  \ref{th:rate_gen} are satisfied. So $(a)-(f)$ and $(a')-(f')$ follow directly from Theorem \ref{th:rate_gen}. 

Now we prove $(g)$. Notice that  $\alpha< 3$,  $\beta=\frac{1}{3}$, and   $\beta_0=\frac{2}{3}$. Consider the functions $\theta(t)$ and $\eta(t)$  defined by \eqref{eq:theta_eta_gen}.
By computations we get
\begin{equation}\label{col:theta_eta}
 \theta(t) = \frac{2\alpha}{3t_0^{\frac{\alpha}{3}}}t^{\frac{\alpha}{3}-1}\quad \text{and}\quad \eta(t) =\frac{2\alpha}{3t_0^{\frac{2\alpha}{3}}}(1-\frac{\alpha}{3})t^{\frac{2\alpha}{3}-2}.
\end{equation}
Then,
\[  \mathcal{E}_{\theta,\eta}^{\beta}(t)\leq  \mathcal{E}_{\theta,\eta}^{\beta}(t_0),\quad\forall t\geq t_0,  \]
where $\mathcal{E}_{\theta,\eta}^{\beta}(t)$ is  the energy function  defined by \eqref{eq:energy_fun_gen}. As a consequence, we have
\begin{equation*}
	\frac{1}{2}\|\theta(t)(x(t)-x^*)+p(t)^{\beta}\dot{x}(t)\|^2+\frac{\eta(t)}{2}\|x(t)-x^*\|^2\leq   \mathcal{E}_{\theta,\eta}^{\beta}(t_0),\quad\forall t\geq t_0.
\end{equation*} 
This implies that for any $t\geq t_0$
\begin{equation}\label{col:gen_eq_x}
\eta(t)\|x(t)-x^*\|^2 \leq   2\mathcal{E}_{\theta,\eta}^{\beta}(t_0)
\end{equation}
and
\begin{equation}\label{col:gen_eq_dotx}
p(t)^{\beta}\|\dot{x}(t)\|-\theta(t)\|x(t)-x^*\| \leq\|\theta(t)(x(t)-x^*)+p(t)^{\beta}\dot{x}(t)\| \leq \sqrt{ 2\mathcal{E}_{\theta,\eta}^{\beta}(t_0)}.
\end{equation}
It follows from \eqref{col:gen_pt}-\eqref{col:gen_eq_dotx} that for any $t\geq t_0$,
\[ t^{\frac{\alpha}{3}-1}\|x(t)-x^*\|\leq \frac{3t_0^{\frac{\alpha}{3}}}{\sqrt{\alpha(3-\alpha)}}\sqrt{\mathcal{E}_{\theta,\eta}^{\beta}(t_0)} \]
and
\[ \left(\frac{t}{t_0}\right)^{\alpha\beta}\|\dot{x}(t)\|   \leq  \sqrt{ 2\mathcal{E}_{\theta,\eta}^{\beta}(t_0)} +  \frac{2\alpha}{3t_0^{\frac{\alpha}{3}}} t^{\frac{\alpha}{3}-1} \|x(t)-x^*\|  \leq (\sqrt{2}+2\sqrt{\frac{\alpha}{3-\alpha}})\sqrt{\mathcal{E}_{\theta,\eta}^{\beta}(t_0)}.
\]
This means
\[\|\dot{x}(t)\|= \mathcal{O}(t^{-\frac{\alpha}{3}}) . \]
By similar arguments, we get 
\[\|\dot{y}(t)\|= \mathcal{O}(t^{-\frac{\alpha}{3}}) \quad\text{and}\quad \|\dot{\lambda}(t)\|= \mathcal{O}(t^{-\frac{\alpha}{3}}).\]
This  proves  $(g)$.
\end{proof}

\begin{remark}
Corollary \ref{cor:cor_alpha_t} improves  \cite[Theorem 3.1 and Theorem 3.2]{ZengJ2019} where convergence rates of a second-order dynamical system based on the primal-dual framework for the problem \eqref{eq:ques} with $g(x)\equiv 0$ and $B = 0$ were established. 
\end{remark}

 \textbf{Case $\gamma(t)=\frac{1}{t (\ln t)^{r}}$ with $r\in[0,1]$}. In this case,
\[\ddot{\gamma}{(t)}=\frac{2(\ln t)^2+3r\ln t+r(r+1)}{t^3(\ln t)^{r+2}}.\]

It is easy to verify that \eqref{eq:inqe_ddot_gamma} holds for all  $\beta \in (0, \frac{1}{3}]$ and $t_0\geq e$. As a consequence of Theorem \ref{th:rate_gen}, we have

\begin{corollary}\label{cor:cor_tlnt}
	Suppose that $\gamma(t)=\frac{1}{t (\ln t)^{r}}$ with $r\in[0,1]$, $\delta(t) = \frac{3t (\ln t)^{r}}{2}$  and $t_0\geq e$. 	Let $(x(t),y(t),\lambda(t))$ be a global solution of the  dynamic \eqref{dy:dy_unpertu} and $(x^*,y^*,\lambda^*)\in \Omega$. Then we have the following results:
	\begin{itemize}
		\item [(a)]	$\mathcal{L}(x(t),y(t),\lambda^*)-\mathcal{L}(x^*,y^*,\lambda^* )= \mathcal{O}(-p(t)^{\frac{2}{3}})$.
		\item[(b)] $ \|Ax(t)+By(t)-b\|= \mathcal{O}(-p(t)^{\frac{1}{3}}).$
		\item [(c)] $\int^{+\infty}_{t_0}\frac{p(t)^{\frac{2}{3}}}{t (\ln t)^{r}}\|Ax(t)+By(t)-b\|^2 dt <+\infty.$
		\item [(d)]  $\int^{+\infty}_{t_0}\frac{p(t)^m}{t (\ln t)^{r}}(\mathcal{L}(x(t),y(t),\lambda^*)-\mathcal{L}(x^*,y^*,\lambda^*))dt <+\infty, \quad \forall m\in(0,\frac{2}{3}).$	
		\item [(e)] $\int^{+\infty}_{t_0}\frac{p(t)^m}{t(\ln t)^{r}}(\|\dot{x}(t)\|^2+\|\dot{y}(t)\|^2+\|\dot{\lambda}(t)\|^2) dt <+\infty,  \quad \forall m\in(0,\frac{2}{3}).$
		\item [(f)] $ \|\dot{x}(t)\|+ \|\dot{y}(t)\|+ \|\dot{\lambda}(t)\| =  \mathcal{O}(p(t)^{-m}),\quad \forall m\in(0,\frac{1}{3}).$
\end{itemize}
Here $p(t) = e^{\int^{\ln t}_{\ln t_0}\frac{1}{s^r}ds} $.
\end{corollary}

\begin{remark}
Corollary \ref{cor:cor_tlnt} can be viewed as analogs of the results  in \cite[Subsection 4.2]{AttouchCCR2018}, where  the convergence rate analysis of $(IGS_{\gamma})$  associated with the unstrained optimization problem\eqref{eq:min_fun} has been discussed when   $\gamma(t)=\frac{1}{t (\ln t)^{r}}$ with $r\in[0,1]$.
\end{remark}

\section{Asymptotic properties for $\gamma(t)=\frac{\alpha}{t^r}$ }

In this section, we investigate the asymptotic behavior of the  dynamic \eqref{dy:dy_unpertu} with $\gamma(t)=\frac{\alpha}{t^r}$, $r\in(-1,1)$ and $\alpha>0$. Let us mention that the results of this section cannot be obtained as applications of the results in Section 2. In the case $\gamma(t)=\frac{\alpha}{t^r}$ with $r\in(-1,1)$ and $\alpha>0$, the condition  \eqref{eq:inqe_ddot_gamma} required in Section 2 is not satisfied. Indeed, in this case, we have $\ddot{\gamma}(t) = \frac{\alpha r(1+r)}{t^{r+2}}$.  If there exists $\beta\in (0,\frac{1}{3}]$ satisfying  \eqref{eq:inqe_ddot_gamma}, i.e.,  $\ddot{\gamma}(t)\geq 2\beta^2\gamma(t)^3$. Then  we have
\[ 0< 2\beta^2\alpha^2\leq \frac{r(r+1)}{t^{2(1-r)}}\to 0,\quad \text{as } t\to+\infty,\]
a contradiction. 

Next, we first prove the existence and uniqueness of  a global solution of the  dynamic \eqref{dy:dy_unpertu} with $\gamma(t)=\frac{\alpha}{t^r}$.

Throughout this section, we always suppose that  $t_0\geq1 $ and $(x^*,y^*,\lambda^*)\in \Omega$. Then  we have $\mathcal{L}(x(t),y(t),\lambda^*)-\mathcal{L}(x^*,y^*,\lambda^* )\geq 0$ for $t\geq t_0$. We consider the energy function ${E}_{\theta,\eta}^{\rho}:[t_0,+\infty)\to[0,+\infty)$ defined by
\begin{equation}\label{eq:energy_tr}
	E_{\theta,\eta}^{\rho}(t) = {E}_0(t)+{E}_1(t)+{E}_2(t)+{E}_3(t),
\end{equation}
where
\begin{equation}\label{eq:energy_tr_sub}
	\begin{cases}
		{E}_0(t) = t^{2\rho}(\mathcal{L}(x(t),y(t),\lambda^*)-\mathcal{L}(x^*,y^*,\lambda^*)	),\\
		{E}_1(t) = \frac{1}{2}\|\theta(t)(x(t)-x^*)+t^\rho\dot{x}(t)\|^2+\frac{\eta(t)}{2}\|x(t)-x^*\|^2,\\
		{E}_2(t) = \frac{1}{2}\|\theta(t)(y(t)-y^*)+t^\rho\dot{y}(t)\|^2+\frac{\eta(t)}{2}\|y(t)-y^*\|^2,\\
		{E}_3(t) = \frac{1}{2}\|\theta(t)(\lambda(t)-x^*)+t^\rho\dot{\lambda}(t)\|^2+\frac{\eta(t)}{2}\|\lambda(t)-\lambda^*\|^2,
	\end{cases}
\end{equation}
$\theta,\eta:[t_0,+\infty)\to [0,+\infty)$ are two suitable functions, and $\rho>0$.
Multiplying the first equation of \eqref{dy:dy_unpertu} with  $\gamma(t)=\frac{\alpha}{t^r}$ by $t^\rho$, we have
\[t^\rho\ddot{x}(t) = -\alpha t^{\rho-r}\dot{x}(t)- t^\rho(\nabla f(x(t))+A^T(\lambda(t)+\delta(t)\dot{\lambda}(t))+A^T(Ax(t)+By(t)-b)). \]
This yields
\begin{eqnarray*}
	&&\dot{{E}}_1(t) = \langle \theta(t)(x(t)-x^*)+t^\rho\dot{x}(t), \dot{\theta}(t)(x(t)-x^*)+\theta(t)\dot{x}(t)+\rho t^{\rho-1}\dot{x}(t)+t^\rho\ddot{x}(t)\rangle \\
	&&\qquad + \frac{\dot{\eta}(t)}{2}\|x(t)-x^*\|^2+ \eta(t) \langle x(t)-x^*,\dot{x}(t)\rangle\\
	&&\quad = \langle \theta(t)(x(t)-x^*)+t^\rho\dot{x}(t), \dot{\theta}(t)(x(t)-x^*)+(\theta(t)+\rho t^{\rho-1}-\alpha t^{\rho-r})\dot{x}(t)\\
	&&\qquad -t^\rho(\nabla f(x(t))+A^T(\lambda(t)+\delta(t)\dot{\lambda}(t))+A^T(Ax(t)+By(t)-b)) \rangle \\
	&&\qquad + \frac{\dot{\eta}(t)}{2}\|x(t)-x^*\|^2+ \eta(t) \langle x(t)-x^*,\dot{x}(t)\rangle\\
	&&\quad= (\theta(t)\dot{\theta}(t)+\frac{\dot{\eta}(t)}{2})\|x(t)-x^*\|^2+ t^\rho(\theta(t)+\rho t^{\rho-1}-\alpha t^{\rho-r})\|\dot{x}(t)\|^2\\
	&&\qquad+(\theta(t)(\theta(t)+\rho t^{\rho-1}-\alpha t^{\rho-r})+\eta(t)+t^\rho\dot{\theta}(t))\langle x(t)-x^*,\dot{x}(t)\rangle \\
	&&\qquad-\theta(t)\delta(t)t^{\rho}\langle Ax(t)-Ax^*, \dot{\lambda}(t)\rangle
	-\delta(t)t^{2\rho}\langle A\dot{x}(t),\dot{\lambda}(t)\rangle\\
	&& \qquad-\theta(t)t^\rho(\langle x(t)-x^*, \nabla f(x(t))+A^T\lambda(t)+A^T(Ax(t)+By(t)-b)\rangle\\
	&&\qquad- t^{2\rho}\langle \dot{x}(t), \nabla f(x(t))+A^T\lambda(t)+A^T(Ax(t)+By(t)-b)\rangle.
	\end{eqnarray*}
Similarly, we have
\begin{eqnarray*}
	&&\dot{{E}}_2(t)= (\theta(t)\dot{\theta}(t)+\frac{\dot{\eta}(t)}{2})\|y(t)-y^*\|^2+ t^\rho(\theta(t)+\rho t^{\rho-1}-\alpha t^{\rho-r})\|\dot{y}(t)\|^2\\
	&&\qquad+(\theta(t)(\theta(t)+\rho t^{\rho-1}-\alpha t^{\rho-r})+\eta(t)+t^\rho\dot{\theta}(t))\langle y(t)-y^*,\dot{y}(t)\rangle \\
	&&\qquad-\theta(t)\delta(t)t^{\rho}\langle By(t)-By^*, \dot{\lambda}(t)\rangle-\delta(t)t^{2\rho}\langle B\dot{y}(t),\dot{\lambda}(t)\rangle\\
	&&\qquad -\theta(t)t^\rho(\langle y(t)-y^*, \nabla g(y(t))+B^T\lambda(t)+B^T(Ax(t)+By(t)-b)\rangle\\
	&&\qquad- t^{2\rho}\langle \dot{y}(t), \nabla g(y(t))+B^T\lambda(t)+B^T(Ax(t)+By(t)-b)\rangle
	\end{eqnarray*}
and 
\begin{eqnarray*}
	&&\dot{{E}}_3(t)= (\theta(t)\dot{\theta}(t)+\frac{\dot{\eta}(t)}{2})\|\lambda(t)-\lambda^*\|^2+ t^\rho(\theta(t)+\rho t^{\rho-1}-\alpha t^{\rho-r})\|\dot{\lambda}(t)\|^2\\
	&&\quad+(\theta(t)(\theta(t)+\rho t^{\rho-1}-\alpha t^{\rho-r})+\eta(t)+t^p\dot{\theta}(t))\langle \lambda(t)-\lambda^*,\dot{\lambda}(t)\rangle \\
	&&\quad +\theta(t)t^p\langle  \lambda(t)-\lambda^*, Ax(t)+By(t)-b\rangle+\theta(t)\delta(t)t^{\rho}\langle \lambda(t)-\lambda^*, A\dot{x}(t)+B\dot{y}(t)\rangle\\
	&&\quad +t^{2\rho}\langle \dot{\lambda}(t),  Ax(t)+By(t)-b\rangle+ \delta(t)t^{2\rho}\langle \dot{\lambda}(t),A\dot{x}(t)+B\dot{y}(t)\rangle.
\end{eqnarray*}
Adding $\dot{{E}}_1(t)$, $\dot{{E}}_2(t)$, $\dot{{E}}_3(t)$ together, using $Ax^*+By^*= b$  and rearranging the terms, we get
\begin{equation*}
		\dot{{E}}_1(t)+	\dot{{E}}_2(t)+	\dot{{E}}_3(t) = \sum^5_{i=1} V_i(t),
\end{equation*}
where
\begin{eqnarray*}
	V_1(t) &=&  \left(\theta(t)\dot{\theta}(t)+\frac{\dot{\eta}(t)}{2}\right)(\|x(t)-x^*\|^2+\|y(t)-y^*\|^2+\|\lambda(t)-\lambda^*\|^2),\\
	V_2(t) &=&(\theta(t)(\theta(t)+\rho t^{\rho-1}-\alpha t^{\rho-r})+\eta(t)+t^\rho\dot{\theta}(t))\\
		&&\times(\langle x(t)-x^*,\dot{x}(t)\rangle+\langle y(t)-y^*,\dot{y}(t)\rangle+\langle \lambda(t)-\lambda^*,\dot{\lambda}(t)\rangle,\\
	V_3(t) &=& -\theta(t)t^\rho(\langle x(t)-x^*, \nabla g(x(t))+A^T\lambda^*\rangle+\langle y(t)-y^*, \nabla g(y(t))+B^T\lambda^*\rangle)\\
		 &&+ \theta(t)\delta(t)t^{\rho}\langle \lambda(t)-\lambda^*, A\dot{x}(t)+B\dot{y}(t)\rangle,\\
	V_4(t) &=& -\theta(t)t^\rho\|Ax(t)+By(t)-b\|^2\\
		 &&+t^\rho(\theta(t)+\rho t^{\rho-1}-\alpha t^{\rho-r}) (\|\dot{x}(t)\|^2+\|\dot{y}(t)\|^2+\|\dot{\lambda}(t)\|^2),\\
	V_5(t) &=& (t^{2\rho}-\theta(t)\delta(t)t^{\rho})\langle \dot{\lambda}(t),Ax(t)+By(t)-b\rangle\\
	\quad && -t^{2\rho}\langle \dot{x}(t),\nabla f(x(t))+A^T\lambda(t)+A^T(Ax(t)+By(t)-b)\rangle\\
	\quad && -t^{2\rho}\langle \dot{y}(t),\nabla g(y(t))+A^T\lambda(t)+B^T(Ax(t)+By(t)-b)\rangle.
\end{eqnarray*}
 Derivate of ${E}_0(t)$ to get
\begin{eqnarray*}
	\dot{{E}}_0(t) &=& {2\rho}t^{2\rho-1}(f(x(t))-f(x^*)+g(y(t))-g(y^*)+\langle \lambda^*,Ax(t)+By(t)-b\rangle)\\
	&&+ {\rho}t^{2\rho-1}\|Ax(t)+By(t)-b\|^2\\
	\quad &&+ t^{2\rho}(\langle \nabla f(x(t)),\dot{x}(t)\rangle+\langle \nabla g(y(t)),\dot{y}(t)\rangle+\langle \lambda^*, A\dot{x}(t)+B\dot{y}(t)\rangle)\\
	&&+t^{2\rho}\langle Ax(t)+By(t)-b,A\dot{x}(t)+B\dot{y}(t)\rangle.
\end{eqnarray*}

The existence and uniqueness of a local solution of the dynamic \eqref{dy:dy_unpertu} can be derived from Proposition \ref{pro:local_exist} when $\gamma(y)=\frac{\alpha}{t^r}$ with $r\in(-1,1)$ and  $\delta(t)$ is locally integrable.  In the following, we will further investigate the existence and uniqueness of  its global solution.

\begin{theorem}\label{th:exist_tr}
		Let $f$ and $g$ be two continuously differentiable functions such that $\nabla f$ and $\nabla g$ are locally Lipschitz continuous. Suppose that  $\gamma(t) = \frac{\alpha}{t^r}$ with $r\in(-1,1)$ and $\delta(t)=\frac{t}{2r_0}$  satisfying:
\begin{itemize}
		\item [(a)] $r_0>\frac{1+r}{2}$ and $\alpha>\max\lbrace 0, (4r_0+r+1)t_0^{r-1}\rbrace $  when $r\in(-1,0]$;
		\item [(b)] $r_0>r$ and $\alpha>\max\lbrace 0 ,(4r_0+2r)t_0^{r-1}\rbrace $  when $r\in(0,1)$.
\end{itemize}

Then  for any initial value $(x_0,y_0,\lambda_0,u_0,v_0,w_0)$, there exists a unique solution $(x(t),y(t),\lambda(t))$ with $x(t)\in\mathcal{C}^2([t_0,+\infty),\mathbb{R}^{n_1})$,  $y(t)\in\mathcal{C}^2([t_0,+\infty),\mathbb{R}^{n_2})$ and  $\lambda(t)\in\mathcal{C}^2([t_0,+\infty),\mathbb{R}^{m})$ of the dynamic \eqref{dy:dy_unpertu} satisfying $(x(t_0),y(t_0),\lambda(t_0))=(x_0,y_0,\lambda_0)$ and $(\dot{x}(t_0), \dot{y}(t_0),\dot{\lambda}(t_0))=(u_0,v_0,w_0)$. 
\end{theorem}
\begin{proof}

By Proposition \ref{pro:local_exist}, there exists a unique solution  $(x(t),y(t),\lambda(t))$  with $x(t)\in\mathcal{C}^2([t_0,T),\mathbb{R}^{n_1})$,  $y(t)\in\mathcal{C}^2([t_0,T),\mathbb{R}^{n_2})$ and  $\lambda(t)\in\mathcal{C}^2([t_0,T),\mathbb{R}^{m})$ of the dynamic \eqref{dy:dy_unpertu} defined on a maximal interval $[t_0,T)$ with $T\leq +\infty$  satisfying the initial condition: $(x(t_0),y(t_0),\lambda(t_0))=(x_0,y_0,\lambda_0)$ and $(\dot{x}(t_0), \dot{y}(t_0),\dot{\lambda}(t_0))=(u_0,v_0,w_0)$.  We shall show $T=+\infty$ in both cases.

Case (a). In this case, in \eqref{eq:energy_tr_sub}, we take $\rho=\frac{r+1}{2}$, 
\begin{equation}\label{eq:theta_eta_rneg}
	\theta(t) = 2r_0 t^{\frac{r-1}{2}}\quad\text{and}\quad \eta(t) =2r_0(\alpha - (2r_0+r)t^{r-1}).
\end{equation}
Clearly,  $\theta(t)>0$ for all $t\geq t_0$. We claim that 
 $$\alpha-(2r_0+r)t^{r-1} \geq \frac{\alpha}{2},\quad\forall t\ge t_0,$$
which yields
\begin{equation}\label{eq:ineq_eta_rneg}
	\eta(t)\geq r_0\alpha>0,\quad \forall t\geq t_0.
\end{equation}
Indeed, we have $\alpha-(2r_0+r)t^{r-1} \geq \frac{\alpha}{2}$  for all $t\ge t_0$ when $2r_0+r<0$.  When $2r_0+r\geq 0$, we have
 $\alpha-(2r_0+r)t^{r-1} \geq \alpha-(2r_0+r)t_0^{r-1}\geq\frac{\alpha}{2}$  for all $t\ge t_0$ since   $\alpha>(4r_0+r+1)t_0^{r-1}>(4r_0+2r)t_0^{r-1}$.

Next we shall prove that the energy function $E_{\theta,\eta}^{\rho}(t)$ defined by \eqref{eq:energy_tr} is nonincreasing on $[t_0,T)$. By computations, we have 
\[\dot{\theta}(t)=r_0(r-1)t^{\frac{r-3}{2}},\quad\text{and}\quad \dot{\eta}(t) = - (4r_0^2+2r_0r)(r-1)t^{r-2}.\]
This together with \eqref{eq:theta_eta_rneg} and $r\in (-1,0]$ yields
\begin{equation}\label{eq:rneg_cond_1}
	\theta(t)\dot{\theta}(t)+\frac{\dot{\eta}(t)}{2} = -r r_0(r-1)t^{r-2}\leq 0
\end{equation}
and 
\begin{equation}\label{eq:rneg_cond_2}
	\theta(t)(\theta(t)+\rho t^{\rho-1}-\alpha t^{\rho-r})+\eta(t)+t^\rho\dot{\theta}(t) = 0.
\end{equation}
Since $\alpha>(4r_0+r+1)t_0^{r-1} $,  we have
\begin{eqnarray}\label{eq:rneg_cond_3}
&&\theta(t)+\rho t^{\rho-1}-\alpha t^{\rho-r}= (2r_0+\frac{r+1}{2})t^\frac{r-1}{2}-\alpha t^{\frac{1-r}{2}} \nonumber \\
&&\qquad =  t^{\frac{1-r}{2}}(\frac{1}{2}((4r_0+r+1)t^{r-1}-\alpha)-\frac{\alpha}{2})\nonumber \\
	&&\qquad \leq  t^{\frac{1-r}{2}}(\frac{1}{2}((4r_0+r+1)t_0^{r-1}-\alpha)-\frac{\alpha}{2})\\
	&&\qquad < - \frac{\alpha}{2}t^{\frac{1-r}{2}}.\nonumber
\end{eqnarray}
From \eqref{eq:theta_eta_rneg} we get
\begin{equation}\label{eq:rneg_cond_4}
	 t^{2\rho}-\theta(t)\delta(t)t^\rho = t^{r+1}-2r_0t^{\frac{r-1}{2}}\times\frac{t}{2r_0}\times t^{\frac{r+1}{2}} = 0.
\end{equation}

By \eqref{eq:rneg_cond_1} and  \eqref{eq:rneg_cond_2}, ${V}_1(t)\leq 0$ and $ {V}_2(t)=0$ for all $t\geq t_0$.  
Since  $f$ and $g$ are convex, by using  \eqref{eq:rneg_cond_3} and \eqref{eq:rneg_cond_4} and  similar arguments as in  \eqref{eq:energy_der_gen}, we have
\begin{eqnarray}\label{eq:erergy_der_rneg}
	\dot{{E}}_{\theta,\eta}^{\rho}(t)&\leq& \dot{{E}}_0(t)+ {V}_3(t)+{V}_4(t)+{V}_5(t)\nonumber \\
	 &\leq& -r_0t^{r}\|Ax(t)+By(t)-b\|^2 -\frac{\alpha}{2} t (\|\dot{x}(t)\|^2+\|\dot{y}(t)\|^2+\|\dot{\lambda}(t)\|^2)\\
	 && -(2r_0-r-1) t^r(\mathcal{L}(x(t),y(t),\lambda^*)
	 -\mathcal{L}(x^*,y^*,\lambda^*))\nonumber\\\
	 &\leq& 0  \nonumber
	\end{eqnarray}
for $t\in[t_0,T)$. As a consequence, the function ${{E}}_{\theta,\eta}^{\rho}(\cdot)$ is nonincreasing on $[t_0,T)$, and so
\begin{equation*}
	{E}_{\theta,\eta}^{\rho}(t)\leq  {E}_{\theta,\eta}^{\rho}(t_0),\qquad \forall t\in[t_0,T).
\end{equation*} 
From \eqref{eq:energy_tr} we have
\[\frac{1}{2}\|\theta(t)(x(t)-x^*)+t^{\frac{r+1}{2}}\dot{x}(t)\|^2+ \frac{\eta(t)}{2}\|x(t)-x^*\|^2\leq   \mathcal{E}_{\theta,\eta}^{\rho}(t_0),\quad \forall t\in[t_0,T).\]
This implies 
\begin{equation}\label{eq:bound_x_rneg}
\sqrt{\eta(t)}\|x(t)-x^*\| \leq   \sqrt{2{E}_{\theta,\eta}^{\rho}(t_0)},\quad \forall t\in[t_0,T)
\end{equation}
and
\begin{equation}\label{eq:bound_dotx_rneg}
\|\theta(t)(x(t)-x^*)+t^{\frac{r+1}{2}}\dot{x}(t)\| \leq \sqrt{ 2{E}_{\theta,\eta}^{\rho}(t_0)},\quad\forall t\in[t_0,T).
\end{equation}
Combining \eqref{eq:bound_x_rneg} with \eqref{eq:ineq_eta_rneg} we get
\begin{equation}\label{addboundx}
	\|x(t)-x^*\|\leq \sqrt{\frac{2}{r_0\alpha}} \sqrt{{E}^\rho_{\theta,\eta}(t_0)}\qquad \forall t\in[t_0,T). 
\end{equation}
Since  $\dot{\theta}(t)=r_0(r-1)t^{\frac{r-3}{2}} \leq 0$,   $\theta(t)$ is nonincreasing. It follows  from \eqref{eq:bound_dotx_rneg} and \eqref{addboundx}  that for any $t\in [t_0,T)$
\[t^{\frac{r+1}{2}}\|\dot{x}(t)\|\leq \sqrt{2{E}^\rho_{\theta,\eta}(t_0)}+\theta(t)\|x(t)-x^*\|\leq  \sqrt{2{E}^\rho_{\theta,\eta}(t_0)}+\theta(t_0) \sqrt{\frac{2}{r_0\alpha}} {E}_{\theta,\eta}^{\rho}(t_0).\]
This together with  $r\in(-1,0]$ and $t_0\ge 1$ yields
\[\sup_{t\in[t_0,T)}\|\dot{x}(t)\|<+\infty.\]
By similar arguments, we have $\sup_{t\in[t_0,T)}\|\dot{y}(t)\|<+\infty$ and $\sup_{t\in[t_0,T)}\|\dot{\lambda}(t)\|<+\infty$.

Now assume on the contrary $T<+\infty$. Clearly, the trajectory $(x(t),y(t),\lambda(t))$ is  
bounded on $[t_0 , T)$. By  \eqref{dy:dy_unpertu} and  assumption,
 $(\ddot{x}(t),\ddot{y}(t),\ddot{\lambda}(t))$ is bounded on  $[t_0 , T)$. It ensues that the solution  $(x(t),y(t),\lambda(t))$ together with its derivative  $(\dot{x}(t),\dot{y}(t),\dot{\lambda}(t))$ have a limit at $t = T$ and therefore can be continued, a contradiction. 	Thus $T = +\infty$. 		

Case (b). In this case,  in \eqref{eq:energy_tr_sub},  we take  $\rho=r$,
\begin{equation}\label{eq:theta_eta_rpos}
	\theta(t)=2r_0t^{r-1}\quad \text{and} \quad  \eta(t) = 2r_0t^{r-1}((1-2r-2r_0)t^{r-1}+\alpha).
\end{equation}
Clearly, $\theta(t)>0$  for all $t\geq t_0$.  Now we show that
\begin{equation}\label{eq:ineq_eta_rpos}
	\eta(t)\geq  r_0\alpha t^{r-1}>0,\quad\forall  t\geq t_0.
\end{equation}
When $2r_0+2r\leq1$, we have 
$$\eta(t) = 2r_0t^{r-1}((1-2r-2r_0)t^{r-1}+\alpha)\geq 2r_0\alpha t^{r-1}\ge  r_0\alpha t^{r-1}>0.$$
When $2r_0+2r>1$, we have
 $\alpha-(2r_0+2r-1)t^{r-1} \geq \alpha-(2r_0+2r-1)t_0^{r-1}\geq\frac{\alpha}{2}$
since    $\alpha>(4r_0+2r)t_0^{r-1}>(4r_0+4r-2)t_0^{r-1}$. Thus \eqref{eq:ineq_eta_rpos} holds.

By computations, we have 
\[\dot{\theta}(t)=2r_0(r-1)t^{r-2} \]
and 
\[\dot{\eta}(t)=2r_0(r-1)t^{r-2}((2-4r-4r_0)t^{r-1}+\alpha).\]
This together with \eqref{eq:theta_eta_rpos} yields 
\begin{equation}\label{eq:eq_rpos_cond1}
	\theta(t)(\theta(t)+\rho t^{\rho-1}-\alpha t^{\rho-r})+\eta(t)+t^\rho\dot{\theta}(t) = 0
\end{equation}
and 
\begin{equation}\label{fypadd}
	\theta(t)\dot{\theta}(t)+\frac{\dot{\eta}(t)}{2} = r_0(r-1)t^{r-2}((2-4r)t^{r-1}+\alpha ).
\end{equation}
We claim that
\begin{equation}\label{eq:eq_rpos_cond2}
	\theta(t)\dot{\theta}(t)+\frac{\dot{\eta}(t)}{2}<0.
\end{equation}
Indeed, since $\alpha>\max\lbrace 0, (4r_0+2r)t_0^{r-1}\rbrace$,  $r_0>r$ and $r\in(0,1)$, we have $\alpha>(4r-2)t_0^{r-1}$. In the case $r\in(0,\frac{1}{2})$, we get $(2-4r)t^{r-1}+\alpha>0$ while in the case  $r\in[\frac{1}{2},1)$, we get\ $(2-4r)t^{r-1}+\alpha\geq \alpha-(4r-2)t_0^{r-1}>0$. So \eqref{eq:eq_rpos_cond2} follows from \eqref{fypadd}.
Since $r_0>r$ and $\alpha>(4r_0+2r)t_0^{r-1}$ with $t_0>1,r\in(0,1)$, we have
\begin{equation}\label{eq:eq_rpos_cond3}
 \theta(t)+rt^{r-1}-\alpha = (2r_0+r)t^{r-1}-\alpha \leq (2r_0+r)t_0^{r-1}-\alpha  <-\frac{\alpha}{2}, \forall t\geq t_0.
\end{equation}
By computations, we have 
\begin{equation}\label{eq:eq_rpos_cond4}
	 t^{2\rho}-\theta(t)\delta(t)t^\rho = t^{2r}-2r_0t^{r-1}\times\frac{t}{2r_0}\times t^r= 0.
\end{equation}
By  \eqref{eq:eq_rpos_cond1}-\eqref{eq:eq_rpos_cond4} and similar arguments as in (a), we get 
 \begin{eqnarray}\label{eq:energy_dot_rpos}
	&&\dot{{E}}_{\theta,\eta}^{\rho}(t)+r_0t^{2r-1}\|Ax(t)+By(t)-b\|^2 +\frac{\alpha}{2} t^r  (\|\dot{x}(t)\|^2+\|\dot{y}(t)\|^2+\|\dot{\lambda}(t)\|^2)\nonumber\\
	&&\qquad  + 2(r_0-r) t^{2r-1}(\mathcal{L}(x(t),y(t),\lambda^*)-\mathcal{L}(x^*,y^*,\lambda^*))\leq 0,
	\end{eqnarray}
for $t\in[t_0,T)$. This implies that the function ${{E}}_{\theta,\eta}^{\rho}(\cdot)$ is nonincreasing on $[t_0,T)$, and so
\begin{equation*}
	{E}_{\theta,\eta}^{\rho}(t)\leq  {E}_{\theta,\eta}^{\rho}(t_0),\qquad \forall t\in[t_0,T).
\end{equation*} 
By similar arguments as  in (a), we have
\begin{equation}\label{eq:bound_x_rpos}
\sqrt{\eta(t)}\|x(t)-x^*\| \leq   \sqrt{2{E}_{\theta,\eta}^{\rho}(t_0)},\quad \forall t\in[t_0,T)
\end{equation}
and
\begin{equation}\label{eq:bound_dotx_rpos}
\|\theta(t)(x(t)-x^*)+t^{r}\dot{x}(t)\| \leq \sqrt{ 2{E}_{\theta,\eta}^{\rho}(t_0)},\quad\forall t\in[t_0,T).
\end{equation}
Combining \eqref{eq:bound_x_rpos} with \eqref{eq:ineq_eta_rpos} we obtain
 \[t^\frac{r-1}{2}\|x(t)-x^*\|\leq \sqrt{\frac{2}{r_0\alpha}} \sqrt{{E}^\rho_{\theta,\eta}(t_0)},\quad \forall t\in[t_0,T). \]
It follows from \eqref{eq:bound_dotx_rpos} and \eqref{eq:theta_eta_rpos} that for any $t\in[t_0,T)$
\begin{equation}\label{bounddotxb}
t^r\|\dot{x}(t)\|\leq \sqrt{2{E}^\rho_{\theta,\eta}(t_0)}+2r_0t^{r-1}\|x(t)-x^*\|\leq  \sqrt{2{E}^\rho_{\theta,\eta}(t_0)}+\sqrt{\frac{8r_0}{\alpha}} t^\frac{r-1}{2}\sqrt{{E}^\rho_{\theta,\eta}(t_0)}.
\end{equation}
Since  $r\in(0,1)$ and $t_0\geq 1$, we have  $t^\frac{r-1}{2}\leq 1$ and   $t^r\geq 1$ for $t\geq t_0$.  The rest of the proof is same as the one in case (a).
\end{proof}

Next,  we  discuss the asymptotic behavior of  the dynamic \eqref{dy:dy_unpertu} with $\gamma(t)=\frac{\alpha}{t^r}$ and $r\in(-1,1)$. We firs consider the case  $r\in(-1,0]$.

\begin{theorem}\label{th:th_rate_rneg}
	 Suppose that $\gamma(t) =  \frac{\alpha}{t^{r}}$ with $r\in(-1,0]$ and $\alpha>0$, and $\sigma(t)=\frac{t}{2r_0}$ with $r_0>\frac{r+1}{2}$.
	Suppose that  $(x(t),y(t),\lambda(t))$ is a global solution of the  dynamic \eqref{dy:dy_unpertu} and $(x^*,y^*,\lambda^*)\in \Omega$. Then, the following conclusions hold:
	\begin{itemize}
		\item [(a)]$\mathcal{L}(x(t),y(t),\lambda^*)-\mathcal{L}(x^*,y^*,\lambda^*)=\mathcal{O}(t^{-(r+1)}).$
		\item [(b)] $\|Ax(t)+By(t)-b\|=\mathcal{O}(t^{-\frac{r+1}{2}}).$
		\item [(c)]  $\int^{+\infty}_{t_0}t^{r}\|Ax(t)+By(t)-b\|^2 dt <+\infty.$
		\item[(d)] $\int^{+\infty}_{t_0}t^{r}(	\mathcal{L}(x(t),y(t),\lambda^*)-\mathcal{L}(x^*,y^*,\lambda^*)) dt <+\infty.  $
		\item [(e)]  $\int^{+\infty}_{t_0}t(\|\dot{x}(t)\|^2+\|\dot{y}(t)\|^2+\|\dot{\lambda}(t)\|^2) dt <+\infty.$
		\item  [(f)] $\|\dot{x}(t)\|+\|\dot{y}(t)\|+\|\dot{\lambda}(t)\|  =\mathcal{O}(t^{-\frac{r+1}{2}})$ .
\end{itemize}
\end{theorem}
\begin{proof}
Take $\theta(t)$ and $\eta(t)$ as in \eqref{eq:theta_eta_rneg}. Consider the energy function  $E^\rho_{\theta,\eta}$ defined by \eqref{eq:energy_tr} with  $\rho=\frac{r+1}{2}$. Since  $\alpha>0$ and $r\in (-1,0]$, there exists $t_1\geq t_0$ such that $\alpha>(4r_0+r+1)t_1^{r-1}$. From  \eqref{eq:erergy_der_rneg} we get
\begin{eqnarray}\label{eq:energy_th_rneg}
	&&\dot{{E}}_{\theta,\eta}^{\rho}(t)+r_0t^{r}\|Ax(t)+By(t)-b\|^2 +\frac{\alpha}{2} t (\|\dot{x}(t)\|^2+\|\dot{y}(t)\|^2+\|\dot{\lambda}(t)\|^2)\nonumber \\
	 &&\qquad +(2r_0-r-1) t^r(\mathcal{L}(x(t),y(t),\lambda^*)
	 -\mathcal{L}(x^*,y^*,\lambda^*))\leq 0  
	\end{eqnarray}
for any $t\geq t_1$, which implies 
\begin{equation}\label{eq:energy_decrease_th_rneg}
	{E}_{\theta,\eta}^{\beta}(t)\leq  {E}_{\theta,\eta}^{\beta}(t_1),\qquad\forall  t\geq t_1.
\end{equation} This together with \eqref{eq:energy_tr} implies  (a), i.e.,
\begin{equation*}
	\mathcal{L}(x(t),y(t),\lambda^*)-\mathcal{L}(x^*,y^*,\lambda^*)=\mathcal{O}(t^{-(r+1)}).
\end{equation*}
By same arguments as in \eqref{eq:th_gen_constain}, we have (b):
\[ \|Ax(t)+By(t)-b\|= \mathcal{O}(t^{-\frac{r+1}{2}}).
\]
Since $r_0>\frac{r+1}{2}>0$ and $\alpha>0$, integrating the inequality \eqref{eq:energy_th_rneg} on $[t_1,+\infty)$, we have
  \[\int^{+\infty}_{t_1}t^{r}\|Ax(t)+By(t)-b\|^2 dt <+\infty, \]
  \[\int^{+\infty}_{t_1}t(\|\dot{x}(t)\|^2+\|\dot{y}(t)\|^2+\|\dot{\lambda}(t)\|^2) dt <+\infty, \]
  \[\int^{+\infty}_{t_1}t^{r}(	\mathcal{L}(x(t),y(t),\lambda^*)-\mathcal{L}(x^*,y^*,\lambda^*)) dt <+\infty. \]
Since $\alpha>(4r_0+r+1)t_1^{r-1}$, from  \eqref{eq:ineq_eta_rneg} we have
\[\eta(t) \geq r_0\alpha,\quad\forall t\geq t_1.\]
This yields (c)-(e).
As shown in the proof of (a) of Theorem \ref{th:exist_tr}, we have
\[\|x(t)-x^*\|\leq \sqrt{\frac{2}{\eta(t)}}\sqrt{E^\rho_{\theta,\eta}(t_1)}\leq \sqrt{\frac{2}{r_0\alpha}}\sqrt{E^\rho_{\theta,\eta}(t_1)}, \quad \forall t\geq t_1,\]
and 
\[t^{\frac{r+1}{2}}\|\dot{x}(t)\| \leq  \sqrt{2{E}^\rho_{\theta,\eta}(t_1)}+\theta(t_1)\|x(t)-x^*\|,\quad \forall t\geq t_1.\]
This implies
\[ \|\dot{x}(t)\|=\mathcal{O}(t^{-\frac{r+1}{2}}).\]
By similar arguments, we have 
\[ \|\dot{y}(t)\|=\mathcal{O}(t^{-\frac{r+1}{2}})\quad\text{and }\quad \|\dot{\lambda}(t)\|=\mathcal{O}(t^{-\frac{r+1}{2}}).\]
So we have (f):
\[ \|\dot{x}(t)\|+\dot{y}(t)\|+ \|\dot{\lambda}(t)\|=\mathcal{O}(t^{-\frac{r+1}{2}}).\]
\end{proof}

Now we investigate  the asymptotic behavior of  the dynamic \eqref{dy:dy_unpertu} with $\gamma(t)=\frac{\alpha}{t^r}$ and $r \in(0,1)$.

\begin{theorem}\label{th:th_rate_rpos}
	Suppose that $\gamma(t) =  \frac{\alpha}{t^{r}}$ with $r\in(0,1)$ and $\alpha>0$, and $\sigma(t)=\frac{t}{2r_0}$ with $r_0>r$.
	Suppose that $(x(t),y(t),\lambda(t))$ is a global solution of the dynamic \eqref{dy:dy_unpertu} and $(x^*,y^*,\lambda^*)\in \Omega$. Then, the following conclusions hold:
	\begin{itemize}
		\item [(a)]$\mathcal{L}(x(t),y(t),\lambda^*)-\mathcal{L}(x^*,y^*,\lambda^*)=\mathcal{O}(t^{-2r})$.
		\item [(b)] $\|Ax(t)+By(t)-b\|= \mathcal{O}(t^{-r}) $.
		\item [(c)]  $\int^{+\infty}_{t_0}t^{2r-1}\|Ax(t)+By(t)-b\|^2 dt <+\infty.$
		\item [(d)]  $\int^{+\infty}_{t_0}t^r(\|\dot{x}(t)\|^2+\|\dot{y}(t)\|^2+\|\dot{\lambda}(t)\|^2) dt <+\infty.$
  		\item[(e)] $\int^{+\infty}_{t_0}t^{2r-1}(	\mathcal{L}(x(t),y(t),\lambda^*)-\mathcal{L}(x^*,y^*,\lambda^*)) dt <+\infty. $
		\item  [(f)] $\|\dot{x}(t)\|+ \|\dot{y}(t)\|+\|\dot{\lambda}(t)\|=\mathcal{O}(t^{-r})$. 
\end{itemize}
\end{theorem}
\begin{proof}
Take $\theta(t)$ and $\eta(t)$ as in  \eqref{eq:theta_eta_rpos}. Consider the energy function $E^\rho_{\theta,\eta}$ defined by \eqref{eq:energy_tr} with $\rho = r$. Since $\alpha>0$ and $r\in (0,1)$, there exists $t_1\geq t_0$ such that $\alpha>(4r_0+2r)t_1^{r-1}$. Then, it follows from \eqref{eq:energy_dot_rpos} that
\begin{eqnarray}\label{eq:energy_th_dot_rneg}
	&&\dot{{E}}_{\theta,\eta}^{\rho}(t)+r_0t^{2r-1}\|Ax(t)+By(t)-b\|^2 +\frac{\alpha}{2} t^r  (\|\dot{x}(t)\|^2+\|\dot{y}(t)\|^2+\|\dot{\lambda}(t)\|^2)\nonumber\\
	&&\qquad  + 2(r_0-r) t^{2r-1}(\mathcal{L}(x(t),y(t),\lambda^*)-\mathcal{L}(x^*,y^*,\lambda^*))\leq 0
	\end{eqnarray}
for any $t\geq t_1$. This implies
\begin{equation*}\label{eq:energy_th_decrease_rneg}
	{E}_{\theta,\eta}^{\rho}(t)\leq  {E}_{\theta,\eta}^{\rho}(t_1),\qquad \forall t\geq t_1.
\end{equation*}
By same arguments as in the proof of Theorem \ref{th:th_rate_rneg}, we can prove (a) and (b).
Since $r_0>r>0$, and $\alpha>0$, integrating the inequality \eqref{eq:energy_th_dot_rneg} on $[t_1,+\infty)$, we have
  \[\int^{+\infty}_{t_1}t^{2r-1}\|Ax(t)+By(t)-b\|^2 dt <+\infty, \]
  \[\int^{+\infty}_{t_1}t^r(\|\dot{x}(t)\|^2+\|\dot{y}(t)\|^2+\|\dot{\lambda}(t)\|^2) dt <+\infty, \]
  \[\int^{+\infty}_{t_1}t^{2r-1}(	\mathcal{L}(x(t),y(t),\lambda^*)-\mathcal{L}(x^*,y^*,\lambda^*)) dt <+\infty, \]
which implies (c)-(e).  Since $\alpha>(4r_0+2r)t_1^{r-1}$, from  \eqref{eq:ineq_eta_rpos} we have
\[\eta(t) \geq r_0\alpha t^{r-1},\quad \forall t\geq t_1.\]
Then, as shown in the proof of (b) of Theorem \ref{th:exist_tr}, we have
\[ t^{\frac{r-1}{2}}\|x(t)-x^*\|\leq \sqrt{\frac{2}{r_0\alpha}}\sqrt{E^\rho_{\theta,\eta}(t_1)}, \quad \forall t\geq t_1,\]
and 
\[t^r\|\dot{x}(t)\| \leq  \sqrt{2{E}^\rho_{\theta,\eta}(t_1)}+2r_0t^{r-1}\|x(t)-x^*\|,\quad t\geq t_1.\]
This yields
\begin{eqnarray*}
	t^r\|\dot{x}(t)\| \leq  \sqrt{2{E}^\rho_{\theta,\eta}(t_1)}+t^{\frac{r-1}{2}} \sqrt{\frac{2}{r_0\alpha}}\sqrt{E^\rho_{\theta,\eta}(t_1)} \leq  \sqrt{2{E}^\rho_{\theta,\eta}(t_1)}+ \sqrt{\frac{2}{r_0\alpha}}\sqrt{E^\rho_{\theta,\eta}(t_1)}
\end{eqnarray*}
since $t^{\frac{r-1}{2}}\leq 1$ for all $t\geq t_1$ when  $r\in(0,1)$. This yields 
\[ \|\dot{x}(t)\|=\mathcal{O}(t^{-r}).\]
By similar arguments, we have
\[ \|\dot{y}(t)\|=\mathcal{O}(t^{-r})\quad \text{and }\quad\|\dot{\lambda}(t)\|=\mathcal{O}(t^{-r}).\]
This proves $(f)$.
\end{proof}

\begin{remark}
In the case  $\gamma(t) =  \frac{\alpha}{t^{r}}$ with $\alpha>0$, asymptotic behaviors  of  $(IGS)_\gamma$ and $(IGS)_{\gamma,\epsilon}$ have been discussed in \cite{AttouchC2017} and \cite{Balti2016} respectively.
\end{remark}

\section{The perturbed case}
In this section, we analyze the asymptotic behavior of the following inertial primal-dual dynamical system with external perturbations:
\begin{equation}\label{dy:dy_pertu}
	\begin{cases}
		\ddot{x}(t)+\gamma(t)\dot{x}(t) = -\nabla f(x(t))-A^T(\lambda(t)+\sigma(t)\dot{\lambda}(t))-A^T(Ax(t)+By(t)-b)+\epsilon(t),\\
		\ddot{y}(t)+\gamma(t)\dot{y}(t) = -\nabla g(y(t))-B^T(\lambda(t)+\sigma(t)\dot{\lambda}(t))-B^T(Ax(t)+By(t)-b)+\epsilon(t),\\
		\ddot{\lambda}(t)+\gamma(t)\dot{\lambda}(t) = A(x(t)+\sigma(t)\dot{x}(t))+B(y(t)+\sigma(t)\dot{y}(t))-b.
	\end{cases}
\end{equation}

When $\epsilon(t)$ decays rapidly enough to zeros as $t\to+\infty$, we will show that asymptotic properties established in  the pervious sections are preserved.

\begin{theorem}\label{th:rate_pertu_gen}
Let $\gamma:[t_0,+\infty)\to(0,+\infty)$ be a nonincreasing  and twice continuously differentiable function satisfying   \eqref{eq:inqe_ddot_gamma} and  \eqref{eq:int_gamma}, $\delta(t)=\frac{1}{\beta_0\gamma(t)}$ with $\beta_0\in[2\beta,1-\beta]$ and  let
  $\epsilon:[t_0,+\infty)\to\mathbb{R}$ be a locally integrable function such that
			\begin{equation}\label{con-ep-p}
\int^{+\infty}_{t_0}p(t)^{\beta}\|\epsilon(t)\|dt<+\infty,
\end{equation}
where  $p(t) = e^{\int_{t_0}^t\gamma(s) ds}$ is defined in \eqref{eq:pt}. Suppose that $(x(t),y(t),\lambda(t))$ is a global solution of the dynamic \eqref{dy:dy_pertu} and  $(x^*,y^*,\lambda^*)$ $\in\Omega$.  Then, the following conclusions hold:
\begin{itemize}
		\item [(a)]	$\mathcal{L}(x(t),y(t),\lambda^*)-\mathcal{L}(x^*,y^*,\lambda^* )= \mathcal{O}(p(t)^{-2\beta}).$
		\item[(b)] $ \|Ax(t)+By(t)-b\|= \mathcal{O}(p(t)^{-\beta}).	$
		\item [(c)] $\int^{+\infty}_{t_0} p(t)^{2\beta}\gamma(t)\|Ax(t)+By(t)-b\|^2 dt <+\infty. $
\end{itemize}
Moreover, we have the following results:
\begin{itemize}
\item[Case I]: $\beta<\frac{1}{3}$ and $\beta_0\in(2\beta,1-\beta)$.  Then
\begin{itemize}
	 \item [($d$)]  $\int^{+\infty}_{t_0}p(t)^{2\beta}\gamma(t)(\mathcal{L}(x(t),y(t),\lambda^*)-\mathcal{L}(x^*,y^*,\lambda^*))dt <+\infty.$
	\item [($e$)]$\int^{+\infty}_{t_0}p(t)^{2\beta}\gamma(t)(\|\dot{x}(t)\|^2+\|\dot{y}(t)\|^2+\|\dot{\lambda}(t)\|^2) dt <+\infty. $
	\item [($f$)] $\|\dot{x}(t)\|+ \|\dot{y}(t)\|+ \|\dot{\lambda}(t)\| =  \mathcal{O}(p(t)^{-\beta})).$

\end{itemize}
 \item[Case II]: $\beta=\frac{1}{3}$ and  $\beta_0=\frac{2}{3}$. Then for any $ \tau\in(0,\frac{1}{3})$ we have
\begin{itemize}
	 \item [($d'$)]  $\int^{+\infty}_{t_0}p(t)^{2\tau}\gamma(t)(\mathcal{L}(x(t),y(t),\lambda^*)-\mathcal{L}(x^*,y^*,\lambda^*))dt <+\infty.$
	\item [($e'$)]$\int^{+\infty}_{t_0}p(t)^{2\tau}\gamma(t)(\|\dot{x}(t)\|^2+\|\dot{y}(t)\|^2+\|\dot{\lambda}(t)\|^2) dt <+\infty. $	 
	\item [($f'$)] $ \|\dot{x}(t)\|+ \|\dot{y}(t)\|+ \|\dot{\lambda}(t)\| =  \mathcal{O}(p(t)^{-\tau}).$
\end{itemize}
\end{itemize}

\end{theorem}
\begin{proof}
Define the  function  $\mathcal{E}_{\theta,\eta}^{\beta,\epsilon}:[t_0,+\infty)\to\mathbb{R}$ by 
\begin{eqnarray}\label{eq:energy_pertu_gen}
	&&\mathcal{E}_{\theta,\eta}^{\beta,\epsilon}(t) =\mathcal{E}_{\theta,\eta}^{\beta}(t)- \int^t_{t_0}\langle \theta(s)(x(s)-x^*)+p(s)^{\beta}\dot{x}(s),p(s)^\beta\epsilon(s)\rangle ds\nonumber \\
	&&\qquad\quad- \int^t_{t_0}\langle \theta(s)(y(s)-y^*)+p(s)^{\beta}\dot{y}(s),p(s)^\beta\epsilon(s)\rangle ds,
\end{eqnarray}
where $\mathcal{E}_{\theta,\eta}^{\beta}(t)$ is defined by \eqref{eq:energy_fun_gen},  $\theta(t)$ and $\eta(t)$  are taken as in \eqref{eq:theta_eta_gen}, i.e.,
\[ \theta(t) = \beta_0 p(t)^\beta\gamma(t)\quad \text{and}\quad	\eta(t) = -\beta_0 p(t)^{2\beta}((\beta_0+2\beta-1)\gamma(t)^2+\dot{\gamma}(t)).\]
By similar arguments  as in \eqref{eq:energy_der_gen}, we have
\begin{eqnarray}\label{eq:energy_dot_pertu_gen}
	&& \dot{\mathcal{E}}_{\theta,\eta}^{\beta,\epsilon}(t)+ (\beta_0-2\beta)p(t)^{2\beta}\gamma(t)(\mathcal{L}(x(t),y(t),\lambda^*)-\mathcal{L}(x^*,y^*,\lambda^*))\nonumber\\
	&&\quad+\frac{\beta_0}{2} p(t)^{2\beta}\gamma(t)\|Ax(t)+By(t)-b\|^2  \\
	&&\quad +(1-\beta-\beta_0)  p(t)^{2\beta}\gamma(t)(\|\dot{x}(t)\|^2+\|\dot{y}(t)\|^2+\|\dot{\lambda}(t)\|^2) \leq  0\nonumber
\end{eqnarray}
for $t\geq t_0$. Then 
\begin{equation}\label{eq:energy_dcrease_pertu}
 \mathcal{E}_{\theta,\eta}^{\beta,\epsilon}(t)\leq  \mathcal{E}_{\theta,\eta}^{\beta,\epsilon}(t_0),\quad \forall t\geq t_0,
\end{equation}
which gives, by definition of $\mathcal{E}_{\theta,\eta}^{\beta,\epsilon}(\cdot)$
\begin{eqnarray*}
	&&\frac{1}{2}\|\theta(t)(x(t)-x^*)+p(t)^{\beta}\dot{x}(t)\|^2+\frac{1}{2}\|\theta(t)(y(t)-y^*)+p(t)^{\beta}\dot{y}(t)\|^2\\
	&&\leq  \int^t_{t_0}\langle \theta(s)(x(s)-x^*)+p(s)^{\beta}\dot{x}(s)+\theta(s)(y(s)-y^*)+p(s)^{\beta}\dot{s}(s),p(s)^{\beta}\epsilon(s)\rangle ds\\
	&&\quad+\mathcal{E}_{\theta,\eta}^{\beta,\epsilon}(t_0)
\end{eqnarray*}
for any $t\geq t_0$.
Applying triangle inequality and Cauchy-Schwarz inequality, we get
\begin{eqnarray*}
	&&\frac{1}{2}(\|\theta(t)(x(t)-x^*)+p(t)^\beta\dot{x}(t)\|+\|\theta(t)(y(t)-y^*)+p(t)^{\beta}\dot{y}(t)\|)^2\\
	&&\leq2\int^t_{t_0} (\|\theta(s)(x(s)-x^*)+p(s)^\beta\dot{x}(s)\|+\|\theta(s)(y(s)-y^*)+p(s)^{\beta}\dot{y}(s)\|)\|p(s)^{\beta}\epsilon(s)\| ds\\
	&&\quad  +2 | \mathcal{E}_{\theta,\eta}^{\beta,\epsilon}(t_0)|.
\end{eqnarray*}
By Lemma \ref{le:A2}, we obtain
\begin{eqnarray*}
	&&\|\theta(t)(x(t)-x^*)+p(t)^r\dot{x}(t)\|+\|\theta(t)(y(t)-y^*)+p(t)^{\beta}\dot{y}(t)\| \\
	&&\qquad \leq 2\sqrt{ | \mathcal{E}_{\theta,\eta}^{\beta,\epsilon}(t_0)|}+2\int^{t}_{t_0}p(s)^{\beta}\|\epsilon(s)\|ds
\end{eqnarray*}
for any $t\geq t_0$.
This together with \eqref{con-ep-p} implies
\begin{eqnarray*}
\sup_{t\geq t_0}(\|\theta(t)(x(t)-x^*)+p(t)^{\beta}\dot{x}(t)\|+\|\theta(t)(y(t)-y^*)+p(t)^{\beta}\dot{y}(t)\|)< +\infty.
\end{eqnarray*}
From \eqref{eq:energy_pertu_gen}, we have
\begin{eqnarray*}
	&&  \mathcal{E}_{\theta,\eta}^{\beta,\epsilon}(t)+ \int^t_{t_0}\langle \theta(s)(x(s)-x^*)+p(s)^{\beta}\dot{x}(s),p(s)^{\beta}\epsilon(s)\rangle ds\\
	&&\qquad+\int^t_{t_0}\langle \theta(s)(y(s)-y^*)+p(s)^{\beta}\dot{s}(s),p(s)^{\beta}\epsilon(s)\rangle ds = \mathcal{E}_{\theta,\eta}^{\beta}(t)\geq 0, \forall t\geq t_0.
\end{eqnarray*}
This implies
\begin{eqnarray*}
	&& \inf_{t\geq t_0} \mathcal{E}_{\theta,\eta}^{\beta,\epsilon}(t)\geq -\sup_{t\geq t_0}(\|\theta(t)(x(t)-x^*)+p(t)^{\beta}\dot{x}(t)\|+\|\theta(t)(y(t)-y^*)+p(t)^{\beta}\dot{y}(t)\|)\\
	&&\qquad\times\int^{+\infty}_{t_0}p(s)^{\beta}\|\epsilon(s)\|ds>-\infty.
\end{eqnarray*}
This together with \eqref{eq:energy_dcrease_pertu} implies that $ \mathcal{E}_{\theta,\eta}^{\beta,\epsilon}(t)$ is bounded on $[t_0,+\infty)$. The rest of the proof is similar as the one of Theorem \ref{th:rate_gen}, and so we omit it.
%
\end{proof}

\begin{remark}
	The condition \eqref{con-ep-p} assumed in  Theorem \ref{th:rate_pertu_gen} is mild and it has been used in \cite{AttouchCCR2018} for  asymptotic analysis of  $IGS_{\gamma,\epsilon}$. Especially, in the case $\gamma(t)=\frac{\alpha}{t}$ with $\alpha>0$, the condition \eqref{con-ep-p} becomes  $\int^{+\infty}_{t_0} t^p \|\epsilon(t)\|dt<+\infty$ with $p=\min\lbrace 1,\frac{\alpha}{3}\rbrace$, which has been used in \cite{AttouchCP2018} and \cite{AttouchCR2019}.
\end{remark}

\begin{remark}
Suppose that  $\nabla f$ and $ \nabla g$ are locally Lipschitz continuous, $\gamma(t)$, $\delta(t)$, and $\epsilon(t)$ are locally integrable,  By Proposition \ref{pro:local_exist}, there exists a unique local solution $(x(t),y(t),$ $\lambda(t))$ of the dynamic \eqref{dy:dy_pertu}  defined on a maximal interval $[t_0,T)$ with $T\leq +\infty$.  Additionally, suppose that $\gamma(t)$ ,$\delta(t)$, $\epsilon(t)$ satisfy the assumptions of Theorem \ref{th:rate_pertu_gen}. By similar argument as in the proof of Theorem \ref{th:th_exist_gen}, we can prove  $T=+\infty$. So the existence and uniqueness of a global solution of the dynamic \eqref{dy:dy_pertu} is established.
\end{remark}

Similarly, we can extend  the convergence rate results established in Section 3 to the  dynamic \eqref{dy:dy_pertu} with $\gamma(t)=\frac{\alpha}{t^r}$.

\begin{theorem}\label{th:rate_pertu_rneg}
	 Let $t_0\geq 1$,  $\gamma(t) =  \frac{\alpha}{t^{r}}$ with $r\in(-1,0]$ and  $\alpha>0$, $\sigma(t)=\frac{t}{2r_0}$ with $r_0>\frac{r+1}{2}$. Let  $\epsilon:[t_0,+\infty)\to\mathbb{R}$ be a locally integrable  function satisfying 
$$\int^{+\infty}_{t_0} t^\frac{r+1}{2}\|\epsilon(t)\|dt<+\infty.$$ 	Suppose that $(x^*,y^*,\lambda^*)\in \Omega$ and that  $(x(t),y(t),\lambda(t))$ is a global solution of the  dynamic \eqref{dy:dy_pertu}. Then, the following conclusions hold:
	\begin{itemize}
		\item [(a)]$\mathcal{L}(x(t),y(t),\lambda^*)-\mathcal{L}(x^*,y^*,\lambda^*)=\mathcal{O}(t^{-(r+1)}).$
		\item [(b)] $\|Ax(t)+By(t)-b\|=\mathcal{O}(t^{-\frac{r+1}{2}}).$
		\item [(c)]  $\int^{+\infty}_{t_0}t^{r}\|Ax(t)+By(t)-b\|^2 dt <+\infty.$
		\item [(d)]  $\int^{+\infty}_{t_0}t(\|\dot{x}(t)\|^2+\|\dot{y}(t)\|^2+\|\dot{\lambda}(t)\|^2) dt <+\infty.$
   		\item[(e)] $\int^{+\infty}_{t_0}t^{r}(	\mathcal{L}(x(t),y(t),\lambda^*)-\mathcal{L}(x^*,y^*,\lambda^*)) dt <+\infty.  $
		\item  [(f)] $\|\dot{x}(t)\|+\|\dot{y}(t)\|+\|\dot{\lambda}(t)\|  =\mathcal{O}(t^{-\frac{r+1}{2}})$ .
\end{itemize}
\end{theorem}

\begin{theorem}\label{th:rate_pertu_rpos}
	Let $t_0\geq 1$, $\gamma(t) =  \frac{\alpha}{t^{r}}$ with $r\in[0,1)$ and $\alpha>0$, $\sigma(t)=\frac{t}{2r_0}$ with $r_0>r$. Let $\epsilon:[t_0,+\infty)\to\mathbb{R}$ be a locally integrable function satisfying
   \[ \int^{+\infty}_{t_0}t^{r}\|\epsilon(t)\|dt<+\infty.\]
	Suppose that  $(x^*,y^*,\lambda^*)\in \Omega$ and that $(x(t),y(t),\lambda(t))$ is a global solution of the dynamic \eqref{dy:dy_pertu}. Then, the following conclusions hold:
	\begin{itemize}
		\item [(a)]$\mathcal{L}(x(t),y(t),\lambda^*)-\mathcal{L}(x^*,y^*,\lambda^*)=\mathcal{O}(t^{-2r})$.
		\item [(b)] $\|Ax(t)+By(t)-b\|= \mathcal{O}(t^{-r}) $.
		\item [(c)]  $\int^{+\infty}_{t_0}t^{2r-1}\|Ax(t)+By(t)-b\|^2 dt <+\infty.$
		\item [(d)]  $\int^{+\infty}_{t_0}t^r(\|\dot{x}(t)\|^2+\|\dot{y}(t)\|^2+\|\dot{\lambda}(t)\|^2) dt <+\infty.$
  		\item[(e)] $\int^{+\infty}_{t_0}t^{2r-1}(	\mathcal{L}(x(t),y(t),\lambda^*)-\mathcal{L}(x^*,y^*,\lambda^*)) dt <+\infty. $
		\item  [(f)] $\|\dot{x}(t)\|+ \|\dot{y}(t)\|+\|\dot{\lambda}(t)\|=\mathcal{O}(t^{-r})$. 
\end{itemize}
\end{theorem}

\begin{remark}
	 When $\gamma(t)=\frac{\alpha}{t^r}$ with $r\in(0,1)$ and $\alpha>0$, the assumptions
$$\int^{+\infty}_{t_0} t^r\|\epsilon(t)\|dt<+\infty \quad\text{and}\quad \int^{+\infty}_{t_0} t^\frac{r+1}{2}\|\epsilon(t)\|dt<+\infty$$   have been used in \cite{Balti2016} for convergence rate analysis of $IGS_{\gamma,\epsilon}$.  For more results on asymptotic analysis of dynamical systems with perturbations associated with unstrained optimization problems, we refer the reader to \cite{Jendoubi,Haraux2012,Sebbouh2020}.
\end{remark}

\section{Conclusion} 
In this paper, we have proposed an inertial primal-dual dynamical system for a separable convex optimization problem with linear equality constraints. By using the Lyapunov analysis approach, we investigate the convergence rates of the trajectories generated by the dynamical system under different choices of the damping functions. We have also shown that convergence  rate results  established  are preserved when   small  perturbations are added to the inertial primal-dual dynamical system. The results obtained improves the results of  Zeng et al. \cite{ZengJ2019},  where convergence rates of a second-order dynamical system based on the primal-dual framework for the problem \eqref{eq:ques} with $g(x)\equiv 0$ and $B = 0$ were established.  Our main results can be also viewed as analogs of the ones  in \cite{AttouchCCR2018}, where  the convergence rate analysis of $(IGS_{\gamma})$  associated with the unconstrained optimization problem \eqref{eq:min_fun}  were derived.  
\appendix
\section{Some auxiliary results} 

The following lemmas have been used in the analysis of the convergence properties of the dynamical systems.
\begin{lemma}\cite[Theorem 2.1]{AttouchCCR2018} \label{le:A1} 
 Let $t_0\geq 0$, $\gamma:[t_0,+\infty)\to(0,+\infty)$ be a nonincreasing  and twice continuously differentiable function satisfying $ \ddot{\gamma}(t)\geq 2\beta^2 \gamma(t)^3$ for some $\beta>0$. Then 
 $\dot{\gamma}(t)\leq -\beta \gamma(t)^2$.
\end{lemma}
%
 
\begin{lemma} \cite[Lemma A.5]{Brezis1973}\label{le:A2}
	Let $\omega:[t_0,T]\to [0,+\infty)$ be integrable, and $C\geq 0$. Suppose $\mu:[t_0,T]\to R $ is continuous and 
	\[ \frac{1}{2}\mu(t)^2\leq \frac{1}{2}C^2+\int^{t}_{t_0}\omega(s)\mu(s)ds\]
	for all $t\in[t_0,T]$. Then $|\mu(t)|\leq C+\int^t_{t_0}\omega(s)ds$ for all $t\in[t_0,T]$.
\end{lemma}

\bibliographystyle{siam}
\bibliography{references}

\end{document}